\documentclass[12pt,a4paper]{article}%
\usepackage{amsmath}
\usepackage{graphicx}
\usepackage{epsfig}%
\usepackage{amsfonts}%
\usepackage{amssymb}

\usepackage{color}
\usepackage[normalem]{ulem}

\newtheorem{theorem}{Theorem}[section]
\newtheorem{corollary}[theorem]{Corollary}

\newtheorem{definition}[theorem]{Definition}

\newtheorem{lemma}[theorem]{Lemma}
\newtheorem{proposition}[theorem]{Proposition}
\newtheorem{remark}[theorem]{Remark}

\newenvironment{proof}[1][Proof]{\noindent \emph{#1.} }{\hfill \ 
\rule{0.5em}{0.5em}}

\makeatletter\@addtoreset{equation}{section}\makeatother
\makeatletter\@addtoreset{figure}{section}\makeatother
\makeatletter\@addtoreset{table}{section}\makeatother
\begin{document}

\title{Tensor Numerical Approach to Linearized Hartree-Fock Equation 
for Lattice-type and Periodic Systems}

\author{V. KHOROMSKAIA, \thanks{Max-Planck-Institute for
        Mathematics in the Sciences, Inselstr.~22-26, D-04103 Leipzig,
        Germany ({\tt vekh@mis.mpg.de}).} \and
        B. N. KHOROMSKIJ \thanks{Max-Planck-Institute for
        Mathematics in the Sciences, Inselstr.~22-26, D-04103 Leipzig,
        Germany ({\tt bokh@mis.mpg.de}).}
        }
 
\date{}

\maketitle

\begin{abstract}  
This paper introduces and analyses the new grid-based tensor approach  
for approximate solution of the eigenvalue problem for linearized Hartree-Fock
equation applied to the 3D lattice-structured and periodic systems.
The set of localized basis functions over spatial $(L_1,L_2,L_3)$ lattice 
in a bounding box (or supercell) is assembled by multiple replicas of those from the unit cell.
All basis functions and operators are discretized on a global 3D tensor grid in the bounding box
which enables rather general basis sets. 
In the periodic case, the Galerkin Fock matrix is shown to have the three-level 
block circulant structure, that allows the FFT-based diagonalization. 
The proposed tensor techniques manifest the twofold benefits: 
(a) the entries of the Fock matrix are computed by  1D operations using
low-rank tensors represented on a 3D grid, (b) the low-rank tensor structure in 
the diagonal blocks of the Fock matrix in the Fourier space reduces the conventional 
3D FFT to the product of 1D FFTs. 
We describe fast numerical algorithms for the block circulant representation of the core 
Hamiltonian in the periodic setting based on low-rank tensor representation of arising
multidimensional functions. 
Lattice type systems in a box with open boundary conditions
are treated by our previous tensor solver
for single molecules, which makes possible calculations on large $(L_1,L_2,L_3)$ 
lattices due to reduced numerical cost for 3D problems.
The numerical simulations for  box/periodic $(L,1,1)$ 
lattice systems in a 3D rectangular ``tube'' with $L$ up to several hundred
confirm the theoretical complexity bounds for the tensor-structured  eigenvalue solvers
in the limit of large $L$.

\end{abstract}



\section{Introduction}\label{sec:introduct}

The efficient numerical simulation of periodic and perturbed periodic systems is one of the 
most challenging computational tasks in quantum chemistry calculations of 
crystalline, metallic and polymer-type compounds.
The reformulation of the nonlinear Hartree-Fock equation for periodic
molecular systems based on the Bloch theory \cite{Bloch:1925}
has been addressed in the literature for more than forty years ago,
and nowadays there are several implementations mostly relying on the analytic 
treatment of arising integral operators \cite{CRYSTAL:2000,CRYSCOR:12,GAUSS:09}.
The mathematical analysis of spectral problems for PDEs with the
periodic-type coefficients was an attractive topic in the recent decade, see
\cite{CancesDeLe:08,CanEhrMad:2012,Ortn:ArX} and the references therein. 
However, the systematic developments and optimization of the 
basic numerical algorithms in the Hartree-Fock calculations for 
large lattice structured compounds still are largely unexplored. 

Grid-based approaches for single molecules and moderate size systems based on the locally adaptive grids 
and multiresolution techniques have been discussed (see 
\cite{HaFaYaBeyl:04,SaadRev:10,Frediani:13,CanEhrMad:2012,Ortn:ArX,BiVale:11,RahOsel:13} 
and references therein). 

In this paper, we consider the Hartree-Fock equation for extended systems composed 
of atoms or molecules,  determined by means of an $(L_1, L_2, L_3)$ lattice in a box, 
both for open boundary conditions and in the periodic setting (supercell). 
% In the latter case the structure of  the Fock matrix can be treated as a perturbation
% to the case of ideally periodic systems.
The grid-based tensor-structured method is applied 
(see \cite{KhKhFl_Hart:09,VKH_solver:13,KhorSurv:10,VeBoKh:Ewald:14} and references therein) 
to calculate the core Hamiltonian in the localized 
Gaussian-type basis sets living on a box/periodic spatial lattice.
To perform numerical integration by using low-rank tensor formats we represent all basis
functions on the fine global grid covering the whole computational box (supercell).
The Hartree-Fock equation
for periodic systems is reformulated as the eigenvalue problem for large  
block circulant matrices which  are diagonalizable in the Fourier space,
that allows efficient computations on large lattices of size $L=\max\{L_1,L_2,L_3\}$.
In the following we consider the model problem  for the Fock operator confined to the core
Hamiltonian part. 

%  
% formulated as the nonlinear eigenvalue problem for large  block-structured matrices
% which  diagonalizable in the Fourier space.

One of the severe difficulties in the Hartree-Fock calculations for lattice-structured periodic or
box-restricted systems is the computation of 3D lattice sums of a large number of long-distance 
Coulomb interaction potentials.
This problem is traditionally treated by the so-called Ewald-type summation 
techniques \cite{Ewald:27}
combined with the fast multipole expansion or/and FFT methods.
Notice that the traditional approaches for lattice summation by the Ewald-type methods
scale as $O(L^3 \log L)$ at least, for both periodic and box-type lattice sums.
We apply the recent  lattice summation method \cite{VeBoKh:Ewald:14} by assembled rank-structured
tensor decomposition, which reduces the asymptotic cost at this computational step 
to linear scaling in $L$, i.e. $O(L)$.

In the presented approach the Fock matrix is calculated directly by 3D grid-based tensor
numerical methods in the basis set of localized Gaussian-type-orbitals (GTO)
specified by $m_0$ elements in the 3D unit cell \cite{VeKh_Diss:10,VKH_solver:13}.
Hence, we do not impose explicitly the periodicity-like features of the 
solution by means of the approximation ansatz that is normally the case in the Bloch formalism.
Instead, the periodic properties of the considered system appear implicitly through the 
block structure in the Fock matrix. 
%We describe the block structure of the 
%core Hamiltonian part in the Fock matrix $H$ in a box and the periodic setting. 
In periodic case this matrix is proved to inherit 
the three-level symmetric block circulant form, that allows its efficient 
diagonalization in the Fourier basis \cite{KaiSay_book:99,Davis}.
In the case of $d$-dimensional lattice ($d=1,2,3$), the weak overlap between lattice 
translated basis functions improves the block sparsity thus reducing the storage 
cost to $O(m_0^2 L)$, while the FFT-based diagonalization procedure amounts to
$O(m_0^2 L^d \log L)$ operations. 
% Further saving in storage can be achieved by employing the block sparsity in the Fock matrix 
% assuming the weak overlap between lattice translated basis functions. 
Introducing the low-rank tensor structure into the diagonal blocks of 
the Fock matrix represented in the Fourier space, and using the initial block-circulant structure 
it becomes possible to further reduce the numerical costs to linear scaling in $L$, $O(m_0^2 L \log L)$. 
We present numerical tests in the case of a rectangular 
3D ``tube'' composed of $(L, 1, 1)$ cells with $L$ up to several hundred.

In the new approach one can potentially benefit from the 
additional flexibility that allows to treat  
slightly perturbed periodic systems in a straightforward way. % within this discretization scheme. 
Such situations may arise, for example, in the case of finite
extended systems in a box (open boundary conditions) also considered in this  paper,  
or for slightly perturbed periodic compounds, say for quasi-periodic systems with 
vacancies \cite{BGKh:12}. 
The proposed numerical scheme can be applied in the framework of 
self-consistent Hartree-Fock calculations, in particular, in the reduced 
Hartree-Fock model \cite{CancesDeLe:08},
where the similar block-structure in the Fock matrix can be observed.
The Wannier-type basis  constructed by 
the lattice translation of the initial localized molecular orbitals precomputed 
on the reference unit cell, can be also adapted to our framework.

Furthermore, the arising block-structured matrix representing the stiffness matrix 
$H$ of the core Hamiltonian, as well as some auxiliary function-related tensors, can be shown
to be well suited for further optimization by imposing the low-rank tensor formats, and in particular, 
the quantics-TT (QTT) tensor approximation \cite{KhQuant:09} of long vectors, 
which especially benefits in the limiting case  of large $L$-periodic systems. 
%We apply efficient $O(\log n)$ quantics method for 
In the QTT approach the algebraic operations on the 
3D $n\times n\times n$ Cartesian grid can be implemented with logarithmic cost $O(\log n)$.
Literature surveys on tensor algebra and rank-structured tensor methods for multi-dimensional PDEs can be found in 
\cite{Kolda,KhorSurv:10,GraKresTo:13}, see also \cite{HaKhSaTy:08,DoKhSavOs_mEIG:13}.

The rest of the paper is organized as follows. 
Section \ref{sec_MLBlock-circ} recalls the main properties of the multilevel block 
circulant matrices with special focus on their diagonalization by FFT. 
Section \ref{sec:core_H} includes the main results on the analysis of
core Hamiltonian on lattice structured compounds. 
In particular, section \ref{Core_Hamil} describes the tensor-structured calculation of the 
core Hamiltonian for large lattice-type molecular/atomic systems.
We recall tensor-structured calculation of the Laplace operator and
fast summation of lattice potentials by assembled canonical tensors.
The complexity reduction due to low-rank tensor structures in the
matrix blocks is discussed (see Proposition \ref{prop:low_rank_coef}). 
Section \ref{sec:Core_Ham_period_FFT} discusses in detail the block circulant 
structure of the core Hamiltonian and presents numerical illustrations for 
$(L,1,1)$ lattice systems. Appendix
recalls the classical results on the properties of block circulant/Toeplitz matrices.

\section{Diagonalizing  multilevel block circulant matrices}
\label{sec_MLBlock-circ}

The direct Hartree-Fock calculations for lattice structured systems 
in the localized GTO-type basis lead to the 
symmetric block circulant/Toeplitz matrices (see Appendix \ref{sec_Append:block-circ}), 
where the first-level blocks, $A_0,...,A_{L-1}$, may have further block structures.
In particular, the Galerkin approximation of the 3D Hartree-Fock core Hamiltonian 
in periodic setting leads to the symmetric, three-level block circulant matrix.

\subsection{Multilevel block circulant/Toeplitz matrices}
\label{ssec:MLblock-circ}

In this section we consider the extension of (one-level) block circulant matrices
described in Appendix.
First, we recall the main notions of multilevel block circulant (MBC) matrices 
with the particular focus on the three-level case. 
%Given $q\in \mathbb{N}$, and 
%Given a multi-index ${\bf k}=(k_1\; k_2 \; k_3)\in \mathbb{N}^3$, with $k_\ell =1,\ldots, L_\ell$,
Given a multi-index ${\bf L}=(L_1, L_2, L_3)$,
we denote $|{\bf L}|=(L_1,\, L_2,\,  L_3)$. 
A matrix class ${\cal BC} (d,{\bf L},m_0)$ ($d=1,2,3$)
of $d$-level block circulant matrices can be introduced by the following recursion.
\begin{definition}\label{def:Bcirc}
%Given ${\bf L}=(L,1,1) \in \mathbb{N}^d$, then
For $d=1$, define a class of one-level block circulant matrices by
${\cal BC} (1,{\bf L},m)\equiv {\cal BC} (L_1,m)$ (see Appendix), where ${\bf L}=(L_1,1,1)$.
For $d=2$, we say that a matrix 
$A\in \mathbb{R}^{|{\bf L}|m_0 \times |{\bf L}| m_0}$ belongs to a class
${\cal BC} (d,{\bf L},m_0)$ if
\[
 A = \operatorname{bcirc}(A_1,...,A_{L_1})\quad \mbox{with}\quad 
 A_j\in {\cal BC}(d-1,{\bf L}_{[1]},m_0),\; j=1,...,L_1,
\]
where ${\bf L}_{[1]}=(L_2,L_3)\in \mathbb{N}^{d-1} $. Similar recursion applies to the case $d=3$.
\end{definition}

Likewise to the case of one-level BC matrices (see Appendix), it can be seen that a matrix 
$A \in {\cal BC} (d,{\bf L},m_0)$, $d=1,2,3$, 
of size $|{\bf L}| m_0 \times  |{\bf L}| m_0$ is completely defined (parametrized) by a $d$th order 
matrix-valued tensor 
${\bf A}=[A_{k_1 ... k_d}]$ of size $L_1\times ... \times L_d $,
($k_\ell=1,...,L_\ell$, $\ell=1,...,d$), with $m_0\times m_0$ matrix entries $A_{k_1 ... k_d}$,  
obtained by folding of the generating first column vector in $A$. 
% It has the equivalent matrix 

A diagonalization of a MBC matrix is based on representation via a sequence of cycling permutation matrices 
$\pi_{L_1}, ...,\pi_{L_d}$, $d=1,2,3$. Recall that the $q$-dimensional Fourier transform (FT) 
can be defined via the Kronecker product of the univariate FT matrices (rank-$1$ operator), 
$$
F_{\bf L}=F_{L_1}\otimes \cdots \otimes F_{L_d}.
$$
The block-diagonal form of a MBC matrix is well known in the literature.
Here we prove the diagonal representation in a form that is useful for the description
 of numerical algorithms. 
To that end we generalize the notations ${\cal T}_L$ and $\widehat{A}$ 
(see Section \ref{sec_Append:block-circ}) to the class of multilevel matrices.
We denote by $\widehat{A}\in \mathbb{R}^{|{\bf L}|m_0\times m_0}$ the first 
block column of a matrix $A\in {\cal BC} (d,{\bf L},m_0)$, with a shorthand notation 
$$
\widehat{A}=[A_0,A_1,...,A_{L-1}]^T,
$$ 
% so that the $n\times m\times m$ tensor ${\cal T}_{\bf n} \widehat{A}$ represents slice-wise
% all generating $m\times m$ matrix blocks. Furthermore, 
so that a $|{\bf L}|\times m_0 \times m_0$ tensor ${\cal T}_{\bf L} \widehat{A}$
represents slice-wise all generating $m_0\times m_0$ matrix blocks.
Notice that in the case $m_0=1$, $\widehat{A}\in \mathbb{R}^{|{\bf L}|}$ represents the first column
of the matrix $A$.
Now the Fourier transform $F_{\bf L}$ applies to ${\cal T}_{\bf L} \widehat{A}$ columnwise, 
and the backward reshaping of the resultant tensor, ${\cal T}_{\bf L}'$, returns 
an $|{\bf L}|m_0 \times m_0$ block matrix column. 

\begin{lemma}\label{lem:DiagMLCirc}
A matrix $A\in {\cal BC} (d,{\bf L},m_0)$, is block-diagonalysed  by 
the Fourier transform,
\begin{equation} \label{eqn:DiagMLcirc}
A= (F_{\bf L}^\ast \otimes I_{m_0}) \operatorname{bdiag} \{ \bar{A}_{\bf 0}, \bar{A}_{\bf 1},\ldots , 
\bar{A}_{\bf L-1}\}(F_{\bf L} \otimes I_{m_0}),
\end{equation} 
where
\[
  \left[ \bar{A}_{\bf 0}, \bar{A}_{\bf 1},\ldots , \bar{A}_{\bf L-1}\right]^T = 
 {\cal T}_{\bf L}'(F_{\bf L} ({\cal T}_{\bf L} \widehat{A})).
\]
\end{lemma}
\begin{proof} First, we confine ourself to the case of three-level matrices, i.e. $d=3$, 
and proceed 
\begin{align*}%\label{eqn:MLbcircDiag}
A & = \sum\limits^{L_1 -1}_{k_1=0} \pi_{L_1}^{k_1} \otimes {\bf A}_{k_1} \\  \nonumber 
%       = \sum\limits^{n_1 -1}_{k_1=0} \pi_{n_1}^{k_1}\otimes 
%       (\sum\limits^{n_2 -1}_{k_2=0} \pi_{n_2}^{k_2}\otimes {\bf A}_{k_1 k_2} )\\  \nonumber
  & = \sum\limits^{L_1 -1}_{k_1=0} \pi_{L_1}^{k_1}\otimes 
      (\sum\limits^{L_2 -1}_{k_2=0} \pi_{L_2}^{k_2}\otimes {\bf A}_{k_1 k_2} )=
      \sum\limits^{L_1 -1}_{k_1=0}\sum\limits^{L_2 -1}_{k_2=0}
      \pi_{L_1}^{k_1}\otimes \pi_{L_2}^{k_2}\otimes {\bf A}_{k_1 k_2}  \\ \nonumber 
  & = \sum\limits^{L_1 -1}_{k_1=0}\sum\limits^{L_2 -1}_{k_2=0}\sum\limits^{L_3 -1}_{k_3=0} 
       \pi_{L_1}^{k_1}\otimes \pi_{L_2}^{k_2}\otimes \pi_{L_3}^{k_3}\otimes A_{k_1 k_2 k_3},    \nonumber
\end{align*} 
where ${\bf A}_{k_1}\in \mathbb{R}^{L_2\times L_3 \times m_0\times m_0}$,
${\bf A}_{k_1 k_2} \in  \mathbb{R}^{L_3 \times m_0\times m_0}$
and $A_{k_1 k_2 k_3}\in \mathbb{R}^{m_0\times m_0}$. 

Diagonalizing the periodic shift matrices $\pi_{L_1}^{k_1}, \pi_{L_2}^{k_2}$, and 
$\pi_{L_3}^{k_3}$ via the Fourier transform (see Appendix), we arrive at
\begin{align}\label{eqn:MLbcircDiag2}
A & = (F_{\bf L}^\ast \otimes I_{m_0}) \left[\sum\limits^{L_1 -1}_{k_1=0}\sum\limits^{L_2 -1}_{k_2=0}
      \sum\limits^{L_3 -1}_{k_3=0}
      D_{L_1}^{k_1}\otimes D_{L_2}^{k_2}\otimes D_{L_3}^{k_3}\otimes A_{k_1 k_2 k_3} \right]
 (F_{\bf L}\otimes I_{m_0})
      \\ 
  & = (F_{\bf L}^\ast \otimes I_{m_0}) 
 \mbox{bdiag}_{m_0\times m_0} \{{\cal T}_{\bf L}'(F_{\bf L} ({\cal T}_{\bf L} \widehat{A}))\} 
(F_{\bf L}\otimes I_{m_0}),\nonumber
 \end{align}
with the monomials of diagonal matrices $D_{L_\ell}^{k_\ell}\in \mathbb{R}^{L_\ell \times L_\ell}$, 
$\ell=1,2,3$ are defined by (\ref{eqn:diagshift}).

The generalization to the case $d >3$ can be proven by the similar argument.
\end{proof}

Taking into account representation (\ref{eqn:symBc}),
the multilevel symmetric block circulant matrix can be described in form 
(\ref{eqn:DiagMLcirc}), such that all real-valued diagonal blocks remain symmetric.

Similar to Definition \ref{def:Bcirc}, a matrix class ${\cal BT}_s (d,{\bf L},m_0)$ 
of symmetric $d$-level block Toeplitz matrices can be introduced by the following recursion.
\begin{definition}\label{def:BToepl}
For $d=1$, ${\cal BT}_s (1,{\bf L},m_0)\equiv {\cal BT}_s (L_1,m_0)$ is the class 
of one-level symmetric block circulant matrices with ${\bf L}=(L_1,1,1)$.
For $d=2$ we say that a matrix 
$A\in \mathbb{R}^{|{\bf L}|m \times |{\bf L}| m_0}$ belongs to a class
${\cal BT}_s (d,{\bf L},m_0)$ if
\[
 A= \operatorname{btoepl}_s(A_1,...,A_{L_1})\quad 
 \mbox{with}\quad A_j\in {\cal BT}_s(d-1,{\bf L_{[1]}},m_0),\; j=1,...,L_1.
\]
Similar recursion applies to the case $d=3$.
%where ${\bf n_{[1]}}=(n_2,n_3,...,n_d)\in \mathbb{N}^{d-1}$.
\end{definition}
The following remark compares the properties of circulant and Toeplitz matrices. 
\begin{remark}\label{rem:BToepl}
A block Toeplitz matrix  does not allow diaginalization by FT as it was the case for 
block circulant matrices.
However, it is well known that a block Toeplitz matrix can be extended to the 
double-size (at each level)
block circulant that makes it possible the efficient matrix-vector multiplication, 
and, in particular, the efficient application of power method for finding 
the senior eigenvalues.
\end{remark}

\subsection{Low-rank tensor structure in matrix blocks}\label{ssec:Tensor_bcirc}

In the case $d=3$, the general block-diagonal representation  
(\ref{eqn:DiagMLcirc}) - (\ref{eqn:MLbcircDiag2}) takes form
\[
 A= (F_{\bf L}^\ast \otimes I_{m_0}) (\sum\limits^{L_1 -1}_{k_1=0}\sum\limits^{L_2 -1}_{k_2=0}
      \sum\limits^{L_3 -1}_{k_3=0}
  D_{L_1}^{k_1}\otimes D_{L_2}^{k_2}\otimes D_{L_3}^{k_3}\otimes A_{k_1 k_2 k_3} ) 
  (F_{\bf L}\otimes I_{m_0}),
\]
that allows the reduced storage costs of order $O(|{\bf L}| m_0^2)$, where $|{\bf L}|=L^3$.
For large $L$ the numerical cost may become prohibitive. 
However, the above representation indicates that the further storage and complexity reduction 
becomes possible if
the third-order coefficients tensor ${\bf A}= [A_{k_1 k_2 k_3}]$, $k_\ell=0,...,L_\ell-1$,  
with the matrix entries $A_{k_1 k_2 k_3}\in \mathbb{R}^{m_0\times m_0}$, 
allows some low-rank tensor representation (approximation) in the multiindex ${\bf k}$
described by a small number of parameters.

%and  $A_{\bf k}= [A_{k_1 k_2 ... k_q}]$ with $m\times m$ matrix elements.

To fix the idea, let us assume the existence of rank-$1$ separable matrix factorization,
\[
 A_{k_1 k_2 k_3} = A_{k_1}^{(1)}\odot A_{k_2}^{(2)} \odot A_{k_3}^{(3)},
 \quad A_{k_1}^{(1)}, A_{k_2}^{(2)},A_{k_3}^{(3)} \in \mathbb{R}^{m_0\times m_0},
\quad \mbox{for} \quad k_\ell=0,...,L_\ell-1,
\]
where $\odot$ denotes the Hadamard (pointwise) product of matrices.
The latted representation can be written in 
the factorized tensor-product form
\begin{align*}\label{eqn:bcirc_R1}
& D_{L_1}^{k_1} \otimes D_{L_2}^{k_2}\otimes D_{L_3}^{k_3}\otimes A_{k_1 k_2 k_3} \\
 = &
((D_{L_1}^{k_1}\otimes A_{k_1}^{(1)}) \otimes I_{L_2}\otimes I_{L_3}  )\odot
 (I_{L_1} \otimes (D_{L_2}^{k_2} \otimes A_{k_2}^{(2)})\otimes I_{L_3} ) \odot
 (I_{L_1} \otimes I_{L_2}\otimes (D_{L_3}^{k_3} \otimes A_{k_3}^{(3)})). 
\end{align*}
Given $\ell \in \{1,...,d\}$ and a matrix $A\in \mathbb{R}^{L_\ell \times L_\ell}$, define 
the {\it tensor prolongation} mapping, 
${\cal P}_\ell: \mathbb{R}^{L_\ell\times L_\ell}\to \mathbb{R}^{|{\bf L}|\times |{\bf L}|}$, by
\begin{equation} \label{eqn:Tensor_prolong}
{\cal P}_\ell(A):= \bigotimes_{i=1}^{\ell-1}I_{L_i}\otimes A \bigotimes_{i=\ell+1}^{d}I_{L_i}.
\end{equation}
This leads to the powerful matrix factorization
\begin{align*}%\label{eqn:Abcirc_R1}
 A=& (F_{\bf n}^\ast \otimes I_m) 
 \left[\sum\limits^{L_1 -1}_{k_1=0}  {\cal P}_1(D_{L_1}^{k_1}\otimes A_{k_1}^{(1)}) \odot
 \sum\limits^{L_2 -1}_{k_2=0}   {\cal P}_2(D_{L_2}^{k_2} \otimes A_{k_2}^{(2)}) \odot
 \sum\limits^{L_3 -1}_{k_3=0}  {\cal P}_3(D_{L_3}^{k_3} \otimes A_{k_3}^{(3)})\right]
 (F_{\bf L}\otimes I_{m_0}),\\
=& (F_{\bf L}^\ast \otimes I_m) 
 \left[ {\cal P}_1(\sum\limits^{L_1-1 -1}_{k_1=0} D_{L_1}^{k_1}\otimes A_{k_1}^{(1)}) \odot
    {\cal P}_2(\sum\limits^{L_2 -1}_{k_2=0} D_{L_2}^{k_2} \otimes A_{k_2}^{(2)}) \odot
   {\cal P}_3(\sum\limits^{L_3 -1}_{k_3=0} D_{L_3}^{k_3} \otimes A_{k_3}^{(3)})\right]
 (F_{\bf n}\otimes I_{m_0}),\\
 = & (F_{\bf L}^\ast \otimes I_{m_0})
 \left[ {\cal P}_1(\mbox{bdiag} F_{L_1} A^{(1)}) \odot
  {\cal P}_2(\mbox{bdiag} F_{L_2} A^{(2)})\odot
  {\cal P}_3(\mbox{bdiag} F_{L_3} \otimes A^{(3)})\right]
 (F_{\bf L}\otimes I_{m_0}),
\end{align*}
where the tensor $A^{(\ell)}\in \mathbb{R}^{L_\ell\times m_0 \times m_0}$ 
is defined by concatenation 
$A^{(\ell)}=[A_{0}^{(\ell)},...,A_{L_\ell-1}^{(\ell)}]^T$,
and the tensor prolongation ${\cal P}_\ell$ is defined by (\ref{eqn:Tensor_prolong}).
% where $G_{k_1}=(D_{n_1}^{k_1}\otimes A_{k_1}) \otimes I_{n_2}\otimes I_{n_3}$,
% $G_{k_2}=I_{n_1} \otimes (D_{n_2}^{k_2} \otimes A_{k_2})\otimes I_{n_3}$,
% $G_{k_3}=I_{n_1} \otimes I_{n_2}\otimes (D_{n_3}^{k_3} \otimes A_{k_3})$.
This representation requires only 1D Fourier transforms thus reducing the numerical cost
to 
$$
O(m_0^2 {\sum}_{\ell=1}^d L_\ell \log L_\ell).
$$ 
% with $|{\bf n}|_1=n_1+n_2 + ... + n_q$.
Moreover, and it is even more important, that the eigenvalue 
problem  for the large matrix $A$ now reduces to only 
 $L_1+L_2+L_3 \ll L_1 L_2 L_3$ small $m_0 \times m_0$ matrix eigenvalue problems.

The above block-diagonal representation for $d=3$  
generalizes easily to the case of arbitrary dimension $d$. Finally, we prove the following
general result.
\begin{theorem}\label{thm:tens_FFT}
 Introduce the notation 
$D_{\bf L}^{\bf k}=  D_{L_1}^{k_1}\otimes D_{L_2}^{k_2}\otimes \cdots \otimes D_{L_d}^{k_d}$,
then we have
\[
 A= (F_{\bf L}^\ast \otimes I_{m_0})(\sum\limits^{\bf L -1}_{\bf k=0}
D_{\bf L}^{\bf k}\otimes  A_{\bf k}) (F_{\bf L} \otimes I_{m_0}).
\]
Assume the separability of a tensor $[A_{\bf k}]$ in  the ${\bf k}$ space, 
we arrive at the factorized block-diagonal form of $A$
\[
A= (F_{\bf L}^\ast \otimes I_{m_0})
\left[ {\cal P}_1(F_{L_1} A^{(1)}) \odot
  {\cal P}_2(F_{L_2} A^{(2)})\odot \dots  \odot
  {\cal P}_d (F_{L_d} A^{(d)})\right]
  (F_{\bf L}\otimes I_{m_0}).
\] 
\end{theorem}

% where each factor is represented by the univariate FT,
% \[
% \sum\limits^{n_\ell -1}_{k_\ell=0} D_{n_\ell}^{k_\ell}\otimes A_{k_\ell}
%  = F_{n_\ell} ({\cal T}_{n_\ell} A_{k_\ell}).
% \]

The rank-$1$ decomposition was considered for the ease of exposition only.
The above low-rank representations can be easily generalized to the case of 
canonical or Tucker formats in ${\bf k}$ space (see Proposition \ref{prop:low_rank_coef} below).

Notice that in the practically interesting 3D case
the use of MPS/TT type factorizations does not take the advantage over the Tucker format
since the Tucker and MPS ranks in 3D  appear to be close to each other.
Indeed, the HOSVD for a tensor of order $3$ leads to the same rank estimates for both
the Tucker and MPS/TT tensor formats.

\section{Core Hamiltonian for lattice structured compounds}
\label{sec:core_H}

In this section we  analyze the structure of 
the Galerkin matrix for the core Hamiltonian part in the Fock operator 
%(\ref{eqn:HFcore})
%\[{\cal H}=-\frac{1}{2} \Delta + v_c,\]
with respect to the localized GTO basis replicated over a lattice,
$\{g_m (x) \}_{1\leq m \leq N_b}, x \in {\mathbb{R}^3}$ in a box, or in 
a supercell with the priodic boundary conditions.
%where $v_c(x)$ is given by (\ref{eqn:ElectrostPot}) and $\Delta$ is the 3D Laplacian.

\subsection{The core Hamiltonian in a GTO basis set}
\label{Core_Hamil}

The nonlinear Fock operator ${\cal F}$
in the governing Hartree-Fock eigenvalue problem,   % posed in $\mathbb{R}^3$,
describing the ground state energy for $2N_b$-electron system, is defined by
\[
\left[-\frac{1}{2} \Delta - v_c(x) + 
 \int_{\mathbb{R}^3} \frac{\rho({ y})}{\|{ x}-{ y}\|}\, d{ y}\right] \varphi_i({ x})
- \int_{\mathbb{R}^3} \; \frac{\tau({ x}, { y})}{\|{x} - { y}\|}\, \varphi_i({ y}) d{ y}
%{\cal F} \varphi_i({ x})
%\equiv (-\frac{1}{2} \Delta + v_c + V_H -{\cal K}) \varphi_i({ x})
 = \lambda_i \, \varphi_i({ x}), %\quad
 % \int_{\mathbb{R}^3}\varphi_i\varphi_j=\delta_{ij},
 \quad x\in \mathbb{R}^3,
\] 
where $i =1,...,N_{orb}$. The linear part in the Fock operator is presented by
the core Hamiltonian 
\begin{equation}\label{eqn:HFcore}
{\cal H}=-\frac{1}{2} \Delta - v_c,
\end{equation}
while the nonlinear Hartree potential and exchange operators, 
% \begin{equation}\label{eqn:Fock}
%  {\cal F} :=-\frac{1}{2} \Delta - v_c + 
%  \int_{\mathbb{R}^3} \frac{\rho({ y})}{\|{ x}-{ y}\|}\, d{ y}
% - \int_{\mathbb{R}^3} \; \frac{\tau({ x}, { y})}{\|{x} - { y}\|}\, (\cdot) d{ y},
% \end{equation}
depend on the unknown eigenfunctions (molecular orbitals)
comprising the electron density, $\rho({ y})= 2 \tau(y,y)$, and the density matrix, 
$
\tau(x,y) =\sum\limits^{N_{orb}}_{i=1} \varphi_i (x)\varphi_i (y),\quad x,y\in \mathbb{R}^3,
$ 
respectively. 
The electrostatic potential in the core Hamiltonian is defined by a sum
\begin{equation} \label{eqn:ElectrostPot}
%{\cal H}=-\frac{1}{2} \Delta + v_c,\quad 
v_c(x)= \sum_{\nu=1}^{M}\frac{Z_\nu}{\|{x} -a_\nu \|},\quad
Z_\nu >0, \;\; x,a_\nu\in \mathbb{R}^3,
\end{equation}
where $M$ is the total number of nuclei in the system, 
$a_\nu$,  $Z_\nu$, represent their Cartesian coordinates and the respective charge numbers.
Here $\|\cdot \|$ means the distance function in $\mathbb{R}^3$.
% \[
% v_c(x)=- \sum_{\nu=1}^{M}\frac{Z_\nu}{\|{x} -a_\nu \|},\quad
% Z_\nu >0, \;\; x,a_\nu\in \mathbb{R}^3,
% \]

Given a general set of localized GTO basis functions $\{g_\mu\}$ ($\mu=1,...,N_b$),
the occupied molecular orbitals $\psi_i$ are approximated by
 \begin{equation}\label{expand}
\psi_i=\sum\limits_{\mu=1}^{N_b} C_{\mu i} g_\mu, \quad i=1,...,N_{orb},
\end{equation}
with the unknown coefficients matrix
$C=\{C_{ \mu i} \}\in \mathbb{R}^{N_b \times  N_{orb}}$ obtained as the solution 
of the discretized Hartree-Fock equation with respect to $\{g_\mu\}$, 
and described by $N_b\times N_b$ Fock matrix.
Since the number of basis functions scales cubically in $L$, $N_b = m_0 L^3$, 
the calculation of the Fock matrix may become prohibitive as $L$ increases 
($m_0$ is the number of basis functions in the unit cell).

In what follows we describe the grid-based tensor method for the  
block-structured representation
of the core Hamiltonian in the Fock matrix in a box and in a supercell subject to
the periodic boundary conditions. 
The stiffness matrix $H=\{h_{\mu \nu}\}$ of the core Hamiltonian
(\ref{eqn:HFcore}) is represented by the single-electron integrals, 
 \begin{equation}\label{eqn:Core_Ham}
h_{\mu \nu}= \frac{1}{2} \int_{\mathbb{R}^3}\nabla g_\mu \cdot \nabla g_\nu dx -
\int_{\mathbb{R}^3} v_c(x) g_\mu g_\nu dx, \quad 1\leq \mu, \nu \leq N_b,
\end{equation}
such that the resulting $N_b\times N_b$ Galerkin system of equations
governed by  the reduced Fock matrix $ H$ reads as follows
% The resultant Galerkin system of nonlinear equations
% for the coefficients matrix $C\in \mathbb{R}^{N_b \times N_{orb}}$, 
% and the respective eigenvalues $\Lambda$, reads as
\begin{align*} \label{eqn:HF discr}
  H C &= SC \Lambda, \quad \Lambda= diag(\lambda_1,...,\lambda_{N_{orb}}), \\
  C^T SC   &=  I_N,   \nonumber
\end{align*}
where the mass (overlap) matrix $S=\{s_{\mu \nu} \}_{1\leq \mu, \nu \leq N_b}$, is given by
$
s_{\mu \nu}=\int_{\mathbb{R}^3} g_\mu g_\nu  dx.
$

The numerically extensive part in (\ref{eqn:Core_Ham}) is related to the integration with 
the large sum of lattice translated Newton kernels. 
Indeed, let $M_0$ be the number of nuclei in the unit cell, then
the expensive calculations are due to the summation over $M_0 L^3$ Newton kernels,  
and further spacial integration of this sum with the large set of localized atomic orbitals 
$\{g_\mu\}$, ($\mu=1,...,N_b$), where $N_b$ is of order $m_0 L^3$.

The present approach solves this problem by using the fast and accurate 
grid-based tensor method for evaluation 
of the electrostatic potential $v_c$ defined by  the lattice sum 
in (\ref{eqn:ElectrostPot}),  see \cite{VeBoKh:Ewald:14}, 
and subsequent efficient computation and structural representation of
the stiffness matrix $V_c$,
\[
V_c=[V_{\mu \nu}]:\quad  V_{\mu \nu}= \int_{\mathbb{R}^3} v_c(x) g_\mu g_\nu dx, 
\quad 1\leq \mu, \nu \leq N_b,
\]
by numerical integration by using the low-rank tensor representation on the grid
of all functions involved.

This approach is applicable to the large $L\times L \times L$ lattice.
In the next sections, we show that in the periodic case the resultant stiffness matrix 
$H=\{h_{\mu \nu}\}$ of the core Hamiltonian 
can be parametrized in the form of a symmetric, three-level block circulant matrix.
In the case of lattice system in a box the block structure of $H$ is a small perturbation
of the block Toeplitz matrix.

% we apply the tensor decomposition method to
% the problem of  structured representation (and fast calculation)
% of the core Hamiltonian in the Hartree-Fock 
% equation in the case of large $L$. Indeed, using these decompositions the 
% stiffness matrix $H=\{h_{\mu \nu}\}$ of the core Hamiltonian 
% can be parametrized in the form of a symmetric, block-structured matrix.

\subsection{Low-rank tensor form of the nuclear potential in a box} 
\label{ssec:nuclear}

We consider the nuclear (core) potential operator 
describing the Coulomb interaction of the electrons with the nuclei, see
(\ref{eqn:ElectrostPot}). 
% \begin{equation}\label{eq Vc}
%  V_c(x)= - \sum_{\nu=1}^{M}\frac{Z_\nu}{\|{x} -a_\nu \|},\quad
% Z_\nu >0, \;\; a_\nu\in \mathbb{R}^3.
% \end{equation}
%where $M$ is the number of nuclei, and $a_\nu$,  $Z_\nu$, represent the corresponding coordinates and charge numbers.
In the scaled unit cell  $\Omega=[-b/2,b/2]^3$, 
we introduce the uniform $n \times n \times n$ rectangular Cartesian grid $\Omega_{n}$
with the mesh size $h=b/n$.
Let $\{ \psi_\textbf{i}\}$ be the set of tensor-product piecewise constant basis functions,
%where $\psi_\textbf{i}$ are the tensor-product piecewise polynomials,
%\begin{align} \label{tenfunct}
$  \psi_\textbf{i}(\textbf{x})=\prod_{\ell=1}^d \psi_{i_\ell}^{(\ell)}(x_\ell)$
% with   $x_\ell \in \Omega_\ell$,
%\end{align}
for ${\bf i}=(i_1,i_2,i_3)\in  I \times I \times I  $,
$i_\ell \in I=\{1,...,n\}$.
The Newton kernel is discretized by the projection/collocation method in the form
of a third order tensor of size $n\times n \times n$, 
defined point-wise as
\begin{eqnarray}
% {\bf P} := [p_{i_1  i_2  i_3}] \in \mathbb{R}^{n \times n \times n },
% \quad
\mathbf{P}:=[p_{\bf i}] \in \mathbb{R}^{n\times n \times n},  \quad
%\mathbf{P}:=\pc{p_\tb{i}}_{\tb{i} \in \mathcal{I}}\in \mathbb{R}^{n\times n \times n},  \quad
 p_{\bf i} = 
%p_{i_1  i_2  i_3}=
\int_{\mathbb{R}^3} \frac{\psi_{{\bf i}}({x})}{\|{x}\|} \,\, \mathrm{d}{x},
% \frac{\psi_{i_1 i_2 i_3}(\textbf{x})}{\|\mathbf{x}\|} \,\, \mathrm{d}\tb{x}.
 % \quad \mbox{where}\quad \Omega_\tb{i}=\operatorname{supp}( \psi_\tb{i}).
  \label{galten}
\end{eqnarray}
see \cite{KhKhFl_Hart:09,BeHaKh:08,VeKh_Diss:10,VKH_solver:13}.
Our low-rank canonical decomposition of the $3$rd order tensor $\mathbf{P}$ is based 
on using exponentially convergent $\operatorname*{sinc}$-quadratures for approximation 
of the Laplace-Gauss transform, see \cite{Stenger,GHK:05,HaKhtens:04I}, 
\[
 \frac{1}{z}= \frac{2}{\sqrt{\pi}}\int_{\mathbb{R}_+} e^{- z^2 t^2 } dt,
\]
which can be adapted to the Newton kernel by substitution $z=\sqrt{x_1^2 + x_2^2  + x_3^2}$.
Rational type approximation by exponential sums have been addressed in \cite{Braess:95,Braess:BookApTh}. 
We denote the resultant $R$-term  canonical representation by
\begin{equation} \label{eqn:sinc_general}
    \mathbf{P} \approx  \mathbf{P}_R 
%\sum_{k=-M}^{M} a_k \bigotimes_{\ell=1}^{3}  {\bf b}^{(\ell)}(t_k)
= \sum\limits_{q=1}^{R} {\bf p}^{(1)}_q \otimes {\bf p}^{(2)}_q \otimes {\bf p}^{(3)}_q
\in \mathbb{R}^{n\times n \times n}.
%\quad a_k,\, t_k \in \mathbb{R},
\end{equation}
In a similar way, we also introduce the ``master tensor'',
 $\widetilde{\bf P}_R \in \mathbb{R}^{\widetilde{n}\times \widetilde{n} \times \widetilde{n}}$,
approximating the Newton kernel 
in the extended (accompanying) domain $\widetilde{\Omega} \supset \Omega$,
and associated with the grid parameter $\widetilde{n}=n_0+n$ (say, $n_0=n$),
\begin{equation*} \label{eqn:master_pot}
\widetilde{\bf P}_R= 
\sum\limits_{q=1}^{R} \widetilde{\bf p}^{(1)}_q \otimes 
\widetilde{\bf p}^{(2)}_q \otimes \widetilde{\bf p}^{(3)}_q
\in \mathbb{R}^{\widetilde{n}\times \widetilde{n} \times \widetilde{n}}.
\end{equation*}

The core potential for the molecule is approximated by the canonical tensor
\[
 {\bf P}_{c} = \sum_{\nu=1}^{M_0} Z_\nu {\bf P}_{{c},\nu}\approx \widehat{\bf P}_{c} 
\in \mathbb{R}^{n\times n \times n},
\]
with the rank bound $rank({\bf P}_{c})\leq M_0 R$,
where the rank-$R$ tensor ${\bf P}_{{c},\nu}$ represents the single Coulomb potential 
shifted according to coordinates of the corresponding nuclei, \cite{VeBoKh:Ewald:14},
\begin{equation} \label{eqn:core_tens}
 {\bf P}_{c,\nu} = {\cal W}_{\nu} \widetilde{\bf P}_R =  
\sum\limits_{q=1}^{R} {\cal W}_{\nu}^{(1)} \widetilde{\bf p}^{(1)}_q \otimes 
{\cal W}_{\nu}^{(2)} \widetilde{\bf p}^{(2)}_q 
\otimes {\cal W}_{\nu}^{(3)} \widetilde{\bf p}^{(3)}_q\in \mathbb{R}^{n\times n \times n},
%\quad rank({\bf P}_{c})\leq M_0 R,
\end{equation}
such that every rank-$R$ canonical tensor 
${\cal W}_{\nu} \widetilde{\bf P}_R \in \mathbb{R}^{n\times n \times n}$
% $$
% {\bf P}_{\nu}= {\cal W}_{\nu} \widetilde{\bf P} \in \mathbb{R}^{n\times n \times n}
% $$ 
is thought as a sub-tensor of the master tensor 
obtained by a shift and restriction (windowing) of $\widetilde{\bf P}_R$ onto the $n \times n \times n$ 
grid $\Omega_{n}$ in the unit cell $\Omega$, $\Omega_{n} \subset \Omega_{\widetilde{n}}$. 
A shift from the origin is specified according to the coordinates of the corresponding nuclei, $a_\nu$,
counted in the $h$-units.
Here $\widehat{\bf P}_{c}$ is the rank-$R_c$ ($R_c\leq M_0 R$, actually $R_c \approx R$) 
canonical tensor obtained from ${\bf P}_{c}$ by the rank optimization 
procedure (see \cite{VeBoKh:Ewald:14}, Remark 2.2).

For the tensor representation of the Newton potentials, ${\bf P}_{{c},\nu}$, we make use 
of the piecewise constant discretization on
the equidistant tensor grid, 
%$\omega_{{\bf 3},n}$, (\ref{disc_gauss}), introduced in \S 3, 
where, in general, the univariate grid size $n$ can be noticeably smaller 
than that used for the piecewise linear discretization applied to the Laplace operator.
%\begin{remark}
Indeed, since we use the global basis 
functions for the Galerkin approximation to the eigenvalue problem, 
the grid-based representation of these basis functions can be different 
in the calculation of the kinetic and potential parts in the
Fock operator. The corresponding choice is the only controlled by the 
respective approximation error and by the numerical efficiency depending on 
the separation rank parameters.
%\end{remark}

The error $\varepsilon >0$ arising due to the separable approximation 
of the nuclear potential 
is controlled by the rank parameter $R_{P}= rank({\bf P}_{c})$. Now
letting $rank({\bf G}_m) = R_m$ implies that each matrix element is to be computed with 
linear complexity in $n$, $O(R_kR_m R_{P} \, n)$. 
The almost exponential convergence of the rank approximation in $R_{P}$ allows us 
the choice $R_{P}=O(|\log \varepsilon |)$.

Let us discuss the lattice structured systems.
Low-rank tensor decomposition of the Coulomb interaction defined by the large lattice sum is 
proposed in \cite{VeBoKh:Ewald:14}.
Given the potential sum $v_c$ in the scaled unit cell $\Omega=[-b/2,b/2]^3$,   of size 
$b\times b \times b$,  we consider an interaction potential in a symmetric box (supercell)  
$$
\Omega_L =B_1\times B_2 \times B_3,
%:= [-b/2+(-L_1/2 -1)\delta] \times[-b/2+(L_2-1)\delta] \times[-b/2+(L_3-1)\delta],
$$ 
consisting of a union of $L_1 \times L_2 \times L_3$ unit cells $\Omega_{\bf k}$,
obtained from $\Omega$ by a shift proportional to
$ b$ in each variable, and specified by the lattice vector $b {\bf k}$, where
${\bf k}=(k_1,k_2,k_3)\in \mathbb{Z}^3$, $-(L_\ell-1)/2  \leq k_\ell\leq (L_\ell-1)/2 $, 
($\ell=1,2,3$), such that,
without loss of generality, we assume $L_\ell= 2 p_\ell +1, p_\ell\in \mathbb{N}$.
Hence, we have %$B_\ell=\Omega$ if $L_\ell=1 $ and 
$$
B_\ell = \frac{b}{2}[- L_\ell  ,L_\ell ], \quad  \mbox{for} \quad 
L_\ell \in \mathbb{N},
$$ 
where $L_\ell=1$ corresponds to one-layer systems in the respective variable. 
Recall that  $b=n h$, 
where $h$ is the spacial grid size that is the same for all spacial variables.
To simplify the discussion, we often consider the case $L_\ell = L$.
We also introduce the accompanying domain $\widetilde{\Omega}_L$.
% Notice that in periodic setting the symmetric supercell corresponding to the indexing
% $-L_\ell  \leq k_\ell\leq L_\ell $ is the commonly used notation.

In the case of extended system in a box, further called case (B), the summation 
problem for the total potential $v_{c_L}$ is formulated in the box 
$\Omega_L= \bigcup_{k_1,k_2,k_3=-(L-1)/2}^{(L-1)/2} \Omega_{\bf k}$ as well as in the 
accompanying domain $\widetilde{\Omega}_L$. 
On each $\Omega_{\bf k}\subset \Omega_L$, the  potential sum of interest, 
$v_{\bf k}(x)=(v_{c_L})_{|\Omega_{\bf k}}$, %$x\in \Omega_{\bf k}$, 
is obtained by summation over all unit cells in $\Omega_L$,
\begin{equation}\label{eqn:EwaldSumE}
v_{\bf k}(x)=  \sum_{\nu=1}^{M_0} \sum\limits_{k_1,k_2,k_3=-(L-1)/2}^{(L-1)/2} 
\frac{Z_\nu}{\|{x} -a_\nu (k_1,k_2,k_3)\|}, \quad x\in \Omega_{\bf k}, 
%\quad \forall \Omega_{\bf k}\subset \Omega_L,
\end{equation}
where $a_\nu (k_1,k_2,k_3)=a_\nu  + b {\bf k}$.
This calculation is performed at each of $L^3$ elementary cells 
$\Omega_{\bf k}\subset \Omega_L$, which is implemented by the tensor summation method
described in \cite{VeBoKh:Ewald:14}. The resultant lattice sum is
presented by the canonical tensor ${\bf P}_{c_L}$ with the rank 
$R_0 \leq M_0 R$,
\begin{equation}\label{eqn:EwaldTensorGl}
{\bf P}_{c_L}= \sum\limits_{\nu=1}^{M_0} Z_\nu \sum\limits_{q=1}^{R}
(\sum\limits_{k_1=0}^{L-1}{\cal W}_{\nu({k_1})} \widetilde{\bf p}^{(1)}_{q}) \otimes 
(\sum\limits_{k_2=0}^{L-1} {\cal W}_{\nu({k_2})} \widetilde{\bf p}^{(2)}_{q}) \otimes 
(\sum\limits_{k_3=0}^{L-1}{\cal W}_{\nu({k_3})} \widetilde{\bf p}^{(3)}_{q}).
\end{equation}
The numerical cost and storage size are bounded by $O(M_0 R L N_L )$, 
and $O(M_0 R N_L)$, respectively (see \cite{VeBoKh:Ewald:14}, Theorem 3.1),
where  $N_L= nL$. The lattice sum is also computed in  the accompanying domain 
$\widetilde{\Omega}_L$,
$\widetilde{\bf P}_{c_L}$, where the grid size is equal to $N_L +2 n_0$.

The lattice sum in (\ref{eqn:EwaldTensorGl}) converges only conditionally as $L\to \infty$.
This aspect will be addressed in Section (\ref{ssec:Complexity_EigPr}) following the approach 
introduced in \cite{VeBoKh:Ewald:14}.

\subsection{Nuclear potential operator in  a box}
\label{ssec:Core_Ham_gener}

First, consider the case of a single molecule in the unit cell.
Given the GTO-type basis set $\{{g}_k\}$, $k=1,...,m_0$, i.e. $N_b=m_0$, associated with the 
scaled unit cell and extended to the local bounding box $\widetilde{\Omega}$.
The corresponding rank-$1$ coefficients tensors 
${\bf G}_k={\bf g}_k^{(1)}\otimes{\bf g}_k^{(2)} \otimes{\bf g}_k^{(3)}$ 
representing their piecewise 
constant approximations $\{\overline{g}_k\}$ on the fine 
$\widetilde{n}\times \widetilde{n}\times \widetilde{n}$ grid. 
Then the entries of 
the respective Galerkin matrix  for the core potential operator $v_c$ in 
(\ref{eqn:ElectrostPot}), ${V}_c=\{{V}_{km}\}$,
 are represented (approximately) by the following tensor operations,
\begin{equation}  \label{eqn:nuc_pot}
 {V}_{km} \approx \int_{\widetilde{\Omega}_L} V_c(x) \overline{g}_k(x) \overline{g}_m(x) dx \approx 
 \langle {\bf G}_k \odot {\bf G}_m ,   {\bf P}_{c}\rangle =: {v}_{km} , 
\quad 1\leq k, m \leq m_0.
\end{equation}
% where $\{\overline{g}_k\}$ denotes the piecewise constant representations to 
% the respective Galerkin basis functions.

In the case of lattice syastem in a box we define the basis set on a supercell $\Omega_{L}$ 
(and on $\widetilde{\Omega}_L$) 
by translation of the generating basis  by the lattice vector $\delta {\bf k}$, i.e., 
$\{g_{\mu}({x})\} \mapsto \{g_{\mu}({x+\delta {\bf k} })\}$,
where ${\bf k}=(k_1,k_2,k_3)$, $0 \leq k_\ell\leq L_\ell -1$, 
($\ell=1,2,3$), assuming zero extension 
of $\{g_{\mu}({x+\delta {\bf k} })\}$ beyond each local bounding box $\widetilde{\Omega}_{\bf k}$. 
In this construction the total number of basis functions is equal to $N_b=m_0 L_1 L_2 L_3$.
In practically interesting case of localized atomic orbital basis functions, 
the matrix  $V_{c_L}$ exhibits the special block 
sparsity pattern since the effective support of localized atomic orbitals 
associated with every unit cell $\Omega_{\bf k} \subset \widetilde{\Omega}_{\bf k}$ 
overlaps only fixed (small) number of neighboring cells.

In the following, the matrix block entries will be numbered by a pair of multi-indicies, 
$V_{c_L}=\{V_{{\bf k}{\bf m}}\}$, ${\bf k}=(k_1,k_2,k_3)$, where the $m_0\times m_0$ 
matrix block $V_{{\bf k}{\bf m}}$ is defined by
\begin{equation} \label{eqn:nuc_MatrSparsP}
%  \overline{v}_{km}=  \int_{\mathbb{R}^3} v_c(x) \overline{g}_k(x) \overline{g}_m(x) dx 
% \approx  
V_{{\bf k}{\bf m}} = \langle {\bf G}_{\bf k} \odot {\bf G}_{\bf m} ,   {\bf P}_{c_L}\rangle, 
\quad - L/2 \leq k_\ell, m_\ell \leq L/2,\quad \ell=1,2,3,
\end{equation}
where the canonical tensors ${\bf G}_{\bf k}$ inherit the same block numbering. 

We denote by $L_0$ the number of 
cells measuring the overlap in basis functions in each spacial direction (overlap constant).

\begin{lemma}\label{lem:SparseCaseE}
Assume that  the number of overlapping cells in each spacial direction does 
not exceed $L_0$, 
then in case (B): (a) the number of non-zero blocks in each block row (column) of the symmetric 
Galerkin matrix $V_{c_L}$ does not exceed $(2 L_0 + 1)^3$, 
(b) the required storage is bounded by $m_0^2 [(L_0 + 1)L]^3$.
\end{lemma}
\begin{proof}
In case (B), the matrix representation $V_{c_L}=\{v_{km}\}\in \mathbb{R}^{N_b\times N_b}$ 
of the tensor as in (\ref{eqn:nuc_pot}) is obtained elementwise by the following 
tensor operations
 \begin{equation} \label{nuc_potMatrTot}
 \overline{v}_{km}=  \int_{\mathbb{R}^3} v_c(x) \overline{g}_k(x) \overline{g}_m(x) dx 
\approx  \langle {\bf G}_k \odot {\bf G}_m ,   {\bf P}_{c_L}\rangle =: v_{km}, 
\quad 1\leq k, m \leq N_b,
\end{equation}
where $\{\overline{g}_k\}$ denotes the piecewise constant representations to 
the respective Galerkin basis functions.
This leads to the final expression
\[
\begin{split}
 {v}_{km} & = \sum\limits_{\nu=1}^{M_0} Z_\nu \sum\limits_{q=1}^{R_{\cal N}}
\langle {\bf G}_k \odot {\bf G}_m , 
(\sum\limits_{k_1=0}^{L_1-1} {\cal W}_{\nu({k_1})} \widetilde{\bf p}^{(1)}_{q}) \otimes 
(\sum\limits_{k_2=0}^{L_2-1} {\cal W}_{\nu({k_2})} \widetilde{\bf p}^{(2)}_{q}) \otimes 
(\sum\limits_{k_3=0}^{L_3-1} {\cal W}_{\nu({k_3})} \widetilde{\bf p}^{(3)}_{q})  \rangle \\
 &=
\sum\limits_{\nu=1}^{M_0} Z_\nu \sum\limits_{q=1}^{R_{\cal N}}
\prod\limits_{\ell=1}^3
\langle {\bf g}_k^{(\ell)} \odot { \bf g}_m^{(\ell)},
\sum\limits_{k_\ell=1}^{L_\ell} {\cal W}_{\nu({k_\ell})} \widetilde{\bf p}^{(\ell)}_{q} \rangle.  
% \langle {G}_k^{(2)} \odot { G}_m^{(2)},(\sum\limits_{k_2=1}^N {\cal W}_{\nu({k_2})} N^{(2)}_{q}) \rangle 
% \langle {G}_k^{(3)} \odot { G}_m^{(3)},(\sum\limits_{k_3=1}^N{\cal W}_{\nu({k_3})} N^{(3)}_{q})  \rangle
\end{split}
\]
Taking into account the block representation (\ref{eqn:nuc_MatrSparsP}) and the overlapping property
\begin{equation} \label{eqn:Overlap_Basis}
 {\bf G}_{\bf k} \odot {\bf G}_{\bf m}=0 \quad \mbox{if} \quad | k_\ell - m_\ell| \geq L_0, 
\end{equation}
we analyze the block sparsity pattern in the Galerkin matrix $V_{c_L}$.
Given $3 M_0 R_{\cal N} $ vectors 
$\sum\limits_{k_\ell=1}^{L_\ell} {\cal W}_{\nu({k_\ell})} \widetilde{\bf p}^{(\ell)}_{q}\in \mathbb{R}^{N_L}$, 
where $N_L$ denotes the total number of grid points in $\Omega_L$ in each space variable.
Now the numerical cost to compute ${v}_{km}$ for every fixed index $(k,m)$ is 
estimated by $O(M_0 R_{\cal N} N_L)$ 
indicating linear scaling in the large grid parameter $N_L$ (but not cubic). 

Fixed the row index in $(k,m)$, then (b) follows from the bound on the total number of 
cells $\Omega_{\bf k}$ in the effective integration 
domain in (\ref{nuc_potMatrTot}), that is $(2 L_0 + 1)^3$,  
and the symmetry of $V_{c_L}$.
\end{proof}

Figure \ref{fig:3DCorePerScellErr} illustrates the sparsity pattern of the 
nuclear potentail operator $V_{c_L}$ in the matrix $H$, for $(L, 1, 1)$ lattice 
in a supercell with $L=32$ and $m_0=4$,
corresponding to the overlapping parameter $L_0=3$. One can observe the nearly-boundary effects due to 
the non-equalized  contributions from the left and from the right (supercell in a box). 

% \begin{figure}[htbp]
% \centering
% % \includegraphics[width=4.0cm]{periodic2.eps}
% \includegraphics[width=6.0cm]{VH_P_64_1_1.eps}\quad \quad
% \includegraphics[width=6.0cm]{NM_Lx64_1_1.eps}
% \caption{\small  Block-sparsity in the matrix $V_{c_L}$, in a supercell for $L_0=3, L=32$ (left); 
% Rotated matrix profile (right).}
% \label{fig:3DCoreHamPer}  
% \end{figure}

Figure \ref{fig:3DCorePerScellErr} shows the difference between matrices $V_{c_L}$ in periodic 
(see \S\ref{ssec:Core_Ham_period} for more details) and non-periodic cases. 
The relative norm of the difference is vanishing  if $L\to \infty$.

\begin{figure}[htbp]
\centering
\includegraphics[width=5.0cm]{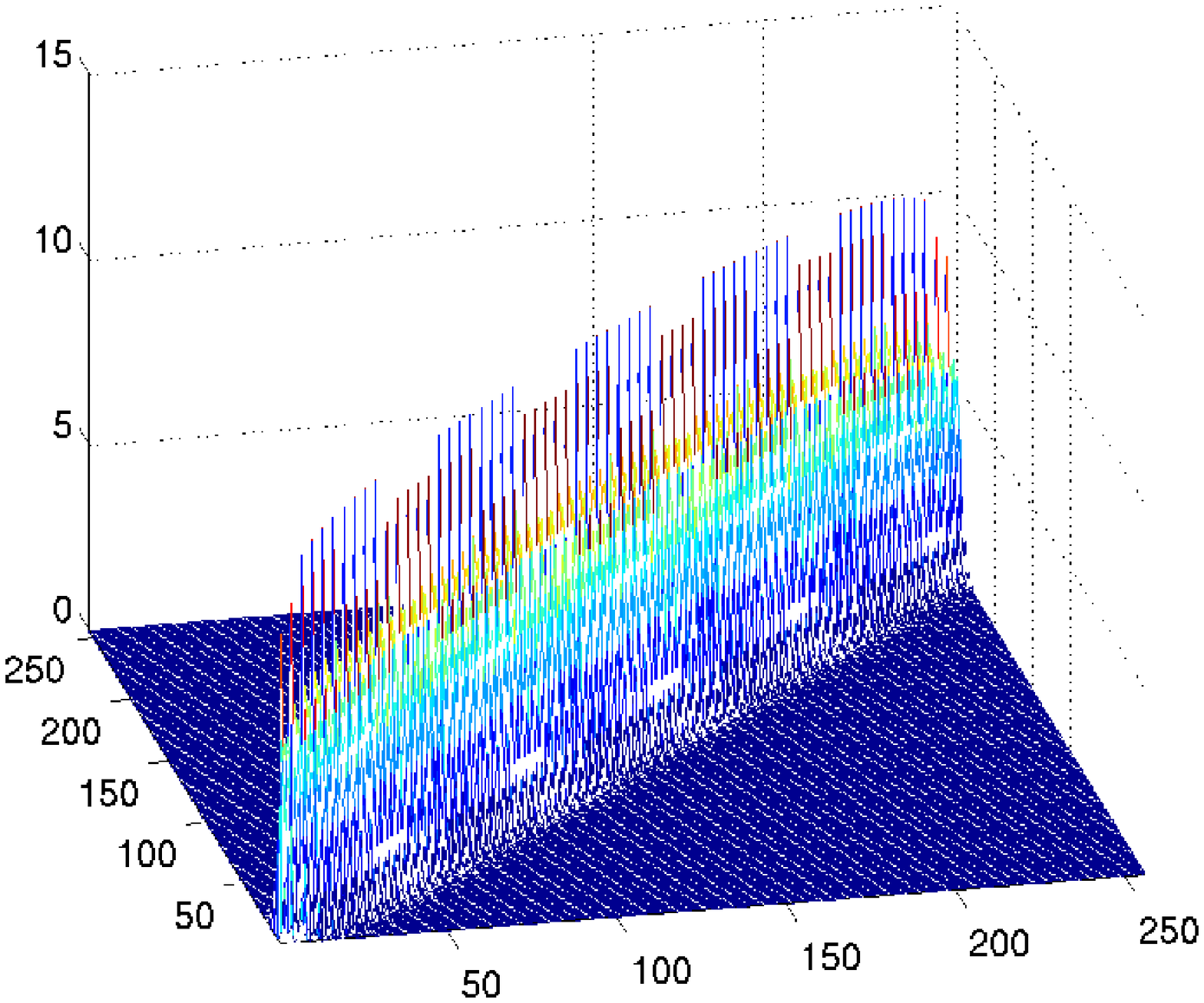}% {VH_P_64_1_1.eps} 
\includegraphics[width=5.0cm]{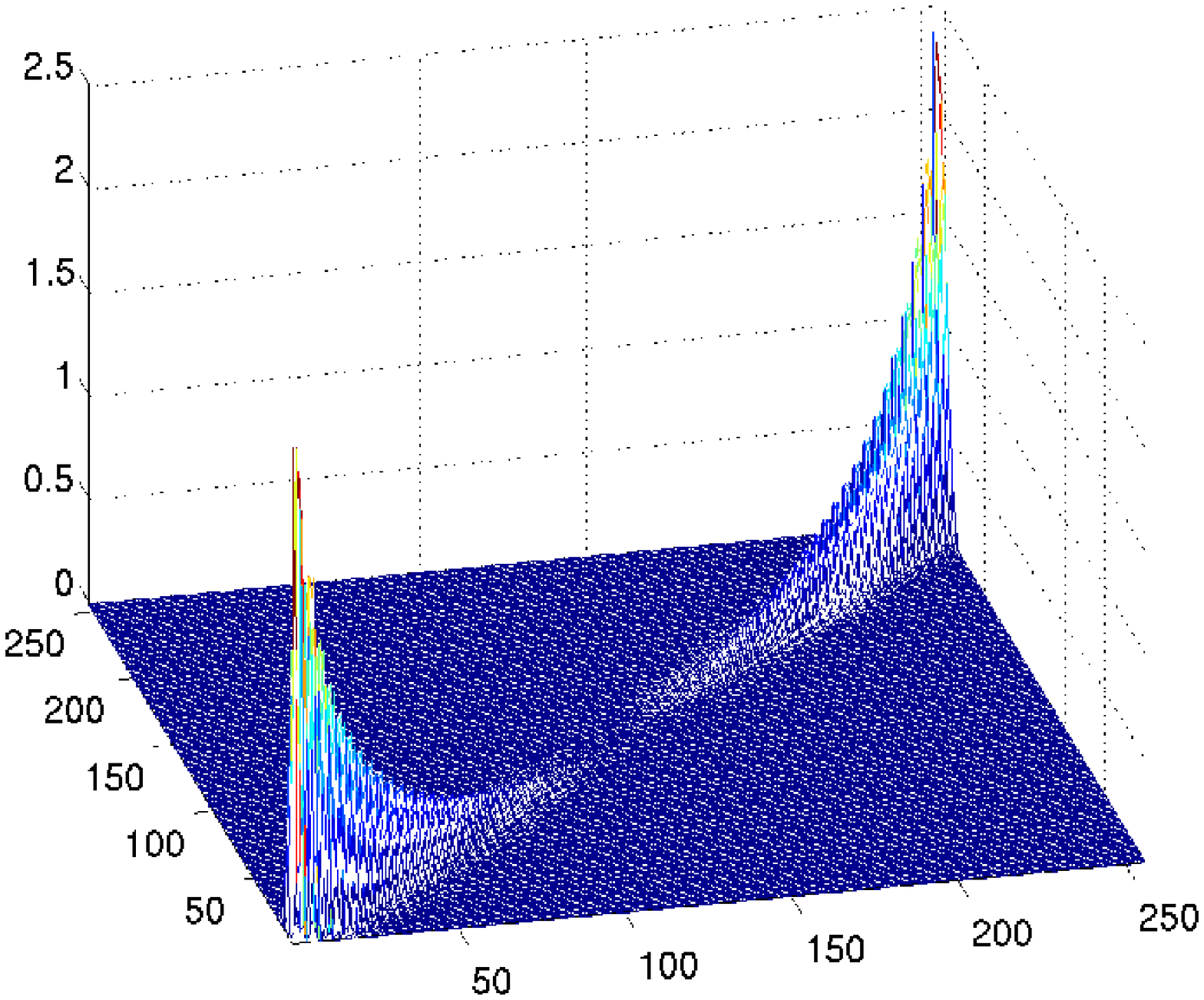} 
\includegraphics[width=5.0cm]{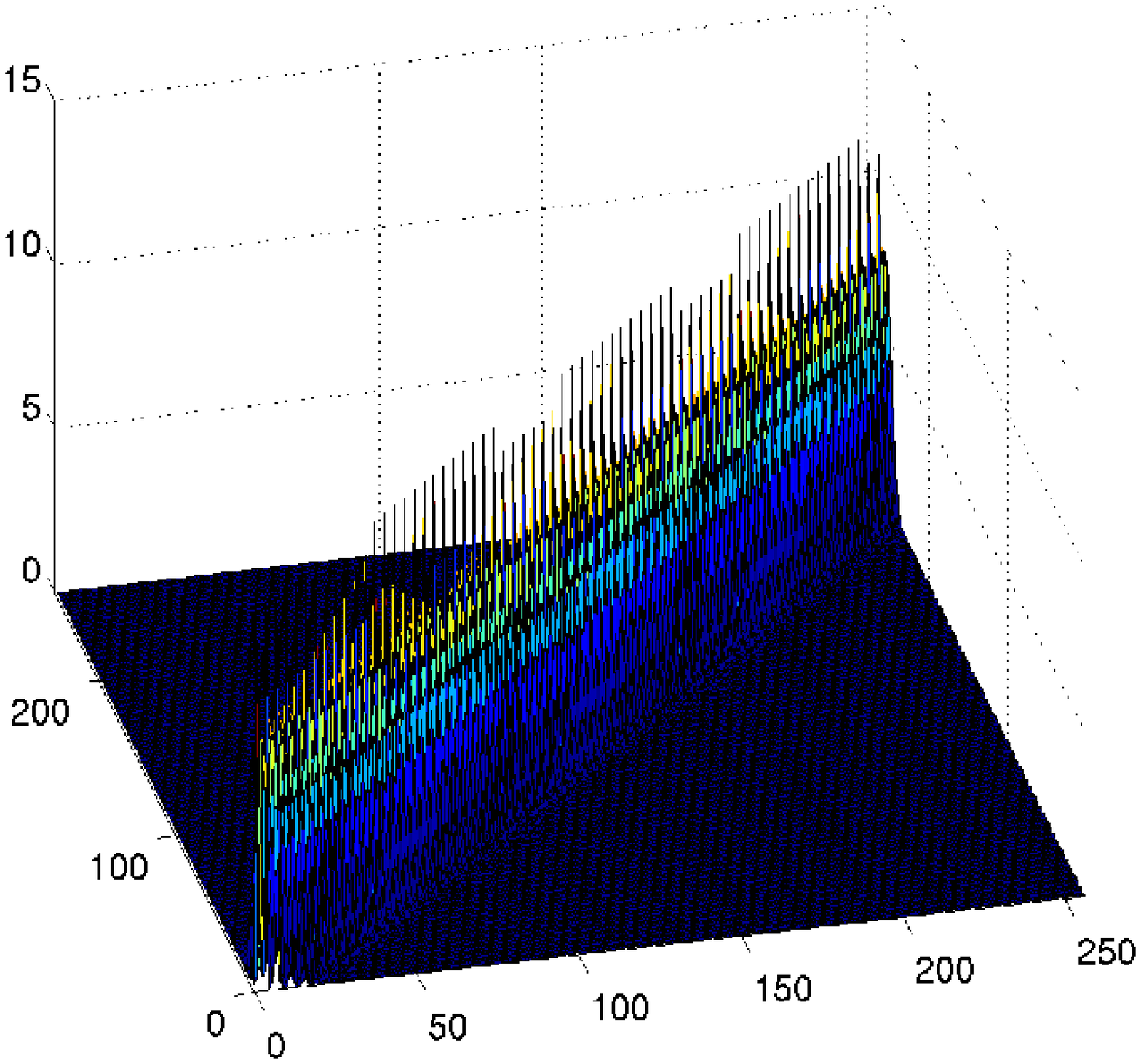}
\caption{\small  
Matrix $V_{c_L}$ in a supercell for $L_0=3, L=32$ (left). 
Difference between  matrices $V_{c_L}$ in periodic 
and single-box cases (middle). Block-sparsity in the matrix $V_{c_L}$ in periodic case (right).}
\label{fig:3DCorePerScellErr}  
\end{figure}

Notice that the quantized approximation of canonical vectors involved in ${\bf G}_k$ 
and ${\bf P}_{c_L}$ reduces
this cost to the logarithmic scale, $O(M_0 R_{\cal N} \log N_L)$, 
that is important in the case of large $L$ in view of $N_L=O(L)$.

The block $L_0$-diagonal structure of the matrix $V_{c_L}=\{V_{{\bf k}{\bf m}}\}$,
${\bf k}\in \mathbb{Z}^3$ ($- L/2 \leq {k}_\ell L/2$)
described by Lemma \ref{lem:SparseCaseE} allows the essential saving in the storage
costs. 

However, the polynomial complexity scaling in $L$ leads to severe limitations on the number of
unit cells. These limitations can be relax if we look more precisely on the
defect between matrix ${V}_{c_L}$ and its block-circulant version corresponding to
the periodic boundary conditions (see \S\ref{ssec:Core_Ham_period}). 
This defect can be split into two components with respect to their local and non-local features:
\begin{enumerate}
 \item [(A)] Non-local effect due to the asymmetry in the interaction potential sum on the lattice in a box.
\item [(B)] The near boundary (local) defect that effects only those blocks in 
$V_{c_L}=\{V_{{\bf k}{\bf m}}\}$ lying in the $L_0$-width of $\partial\Omega_L$,
\[
 L_0+1 -(L-1)/2\leq k_\ell, m_\ell \leq (L-1)/2-1 -L_0.
\] 
\end{enumerate}

Item (A) is related to a 
slight  modification of the core potential to the shift invariant Toeplitz-type form 
$V_{{\bf k}{\bf m}} = V_{|{\bf k}-{\bf m}|}$
by replication of the central block corresponding to $k=0$,
as considered in the next section. In this way the overlap condition 
(\ref{eqn:Overlap_Basis}) for the tensor ${\bf G}_{\bf k}$ 
will impose the block sparsity.

The boundary effect in item (B) becomes relatively small for large number of cells
so that the block-circulant part of the matrix $V_{c_L}$ is getting dominating
as $L\to \infty$. 

The full diagonalization for above mentioned matrices is impossible.
However the efficient storage and fast matrix-times-vector algorithms 
can be applied in the framework of iterative methods for calculation of a small
subset of eigenvalues.

\subsection{Discrete Laplacian  and the mass matrix}\label{ssec:Lap_Op}

The Laplace operator part included in eigenvalue problem for a single molecule 
is posed in the unit cell $ \Omega=[-b/2,b/2]^3 \in \mathbb{R}^3 $,
subject to the homogeneous Dirichlet boundary conditions on $\partial \Omega$.
In periodic case they should be substituted by the periodic boundary conditions.
For given discretization parameter $\overline{n} \in \mathbb{N}$,  we use
the equidistant $\overline{n}\times \overline{n} \times \overline{n}$ tensor 
grid $\omega_{{\bf 3},\overline{n}}=\{x_{\bf i}\} $, 
${\bf i} \in {\cal I} :=\{1,...,\overline{n}\}^3 $,
with the mesh-size $h=2b/(\overline{n} + 1)$,
which might be different from the grid  $\omega_{{\bf 3},n} $ introduced
for representation of the interaction potential (usually, $n\leq \overline{n}$).

Define a set of piecewise linear basis functions 
$\overline{g}_k := {\bf I}_1 g_k $, $k=1,...,N_b$, 
by linear tensor-product interpolation via the set of product hat functions,
 $\{\xi_{\bf i}\}=  \xi_{i_1} (x_1) \xi_{i_2} (x_2)\xi_{i_3} (x_3)$, 
${\bf i} \in {\cal I}$,
associated with the respective grid-cells in $\omega_{{\bf 3},N}$.
Here the linear interpolant ${\bf I}_1= {I}_1\times {I}_1 \times {I}_1$ is 
a product of 1D interpolation operators, $\overline{g}_k^{(\ell)}= {I}_1 {g}_k^{(\ell)}$, $\ell=1,...,3$,
where ${I}_1:C^0([-b,b])\to W_h:=span\{\xi_i\}_{i=1}^{\overline{n}}$ 
is defined over the set of piecewise linear basis functions by 
$$
({I}_1 \, w)(x_\ell):=\sum_{i=1}^{\tilde{n}} w(x_{i_\ell})\xi_{i}(x_\ell), 
\quad  x_{\bf i} \in \omega_{{\bf 3},\tilde{n}}.
$$

With these definitions, 
the rank-$3$ tensor representation of the standard FEM Galerkin stiffness matrix for the Laplacian,
$A_3$, in the tensor basis 
$\{\xi_i(x_1) \xi_j (x_2)\xi_k (x_3) \} $, $i,j,k = 1,\ldots \overline{n}$,
is given by
$$
 A_3 := A^{(1)} \otimes S^{(2)} \otimes S^{(3)} + S^{(1)} \otimes A^{(2)} \otimes S^{(3)}
+ S^{(1)} \otimes S^{(2)} \otimes A^{(3)}\in 
\mathbb{R}^{\overline{n}^{\otimes 3}\times \overline{n}^{\otimes 3}}, 
$$
where the 1D stiffness and mass matrices 
$A^{(\ell)}, S^{(\ell)} \in \mathbb{R}^{\overline{n}\times \overline{n}}$,
$\ell=1,\,2,\,3$, are represented by
\[
 A^{(\ell)} := \{ \langle \frac{d}{d x_\ell} \xi_i(x_\ell) , \frac{d}{d x_\ell} \xi_j(x_\ell) 
\rangle \}^{\overline{n}}_{i,j=1} = \frac{1}{h} \mbox{tridiag} \{-1,2,-1\},
\]
\[ 
 S^{(\ell)}=\{ \langle \xi_i ,\xi_j\rangle \}^{\overline{n}}_{i,j=1} = \frac{h}{6}\;
 \mbox{tridiag} \{1,4,1\},
\]
respectively.

% and $\nabla_{(1)}=\frac{d}{d x_\ell}$.
% Since $\{ \xi_i \}^{\tilde{n}}_{i=1}$ can be chosen  the same for all modes $\ell=1,\, 2,\, 3$, 
% we simplify notations as $A^{(\ell)}=A_1$, and $S^{(\ell)}=S_1$.

This leads to the separable grid-based 
approximation of the initial basis functions $g_k(x)$,
\begin{equation}\label{eq. Gaus pwl}
g_k (x) \approx \overline{g}_k (x)=\prod^3_{\ell=1} 
\overline{g}_k^{(\ell)} (x_{\ell})=\prod^3_{\ell=1} 
\sum\limits^{\overline{n}}_{i=1} g_{k}^{(\ell)}(x_{i_\ell}) \xi_i (x_{\ell}),
\end{equation}
where the rank-$1$ coefficients tensor ${\bf G}_k$ is given by 
${\bf G}_k= {\bf g}_k^{(1)} \otimes {\bf g}_k^{(2)} \otimes {\bf g}_k^{(3)}$, 
with the canonical vectors
${\bf g}_k^{(\ell)}=\{g_{k_i}^{(\ell)}\}\equiv \{g_{k}^{(\ell)}(x_{i_\ell})\}$. 
%(see Figure \ref{fig:Laplace1D} illustrating the construction of $\overline{g}_k (x_1)$).
Let us agglomerate the rank-$1$ tensors ${\bf G}_k\in \mathbb{R}^{{\overline{n}}^{\otimes 3}}$, 
($k=1,...,N_b$) in a tensor-valued matrix $G\in  \mathbb{R}^{N^{\otimes 3}\times N_b}$, 
the Galerkin matrix in the basis set ${\bf G}_k $ can be written in a matrix form
\[
 A_G= G^T A_3 G\in  \mathbb{R}^{N_b \times N_b},
\]
corresponding to the standard matrix-matrix transform under the change of basis.
The matrix entries in $A_G=\{a_{k m}\}$ can be represented by
\[
 a_{k m}= \langle A_3 {\bf G}_k, {\bf G}_m\rangle, \quad  k,m = 1,...,N_b.
\]
Likewise, for the entries of the stiffness matrix we have 
$s_{k m}= \langle {\bf G}_k, {\bf G}_m\rangle$.

It is easily seen that in the periodic case both matrices, $A_G$ and $S$, take the
multilevel block circulant structure.

\section{Linearized  spectral problem by FFT-diagonalization}
\label{sec:Core_Ham_period_FFT}

There are two possibilities for mathematical modeling of the $L$-periodic molecular 
systems, composed, of $(L, L, L)$ elementary unit cells.  
In  the first approach, 
the system is supposed to contain an infinite set of equivalent atoms that 
map identically into itself under any translation by $L$ units in each spacial direction. 
The other model is based on the ring-type periodic structures 
consisting of $L$ identical units in each spacial direction, where every
unit cell of the periodic compound will be mapped to itself by applying a rotational 
transform  from the corresponding rotational group symmetry \cite{SzOst:1996}.

The main difference between these two concepts is in the treatment of the lattice sum 
of Coulomb interactions, thought, in the limit of $L\to \infty$ both models 
approach each other.
In this paper we mainly follow the first approach with the particular 
focus on the asymptotic complexity optimization  
for large lattice parameter $L$. The second concept is useful
for understanding the block structure of the Galerkin matrices for 
Laplacian and the identity operators.

\subsection{Block circulant structure of core Hamiltonian in periodic case}
\label{ssec:Core_Ham_period}

Inthis section we consider the periodic case, further called case (P), and
derive the more refined sparsity pattern of the matrix $V_{c_L}$ 
using the $d$-level ($d=1,2,3$) tensor structure  in this matrix.
The matrix block entries are numbered by a pair of multi-indices, 
$V_{c_L}=\{V_{{\bf k}{\bf m}}\}$, ${\bf k}=(k_1,k_2,k_3)$, where the $m_0\times m_0$ 
matrix block $V_{{\bf k}{\bf m}}$ is defined by (\ref{eqn:nuc_MatrSparsP}).
% \begin{equation} \label{nuc_MatrSparsP}
% %  \overline{v}_{km}=  \int_{\mathbb{R}^3} v_c(x) \overline{g}_k(x) \overline{g}_m(x) dx 
% % \approx  
% V_{{\bf k}{\bf m}} = \langle {\bf G}_{\bf k} \odot {\bf G}_{\bf m} ,   {\bf P}_{c_L}\rangle, 
% \quad 1\leq k_\ell, m_\ell \leq L,\quad \ell=1,2,3.
% \end{equation}
Figure \ref{fig:3DPeriodStruct}  illustrates an example of 3D lattice-type structure of 
size $(4, 4, 2)$. 

\begin{figure}[htbp]
\centering
\includegraphics[width=7.0cm]{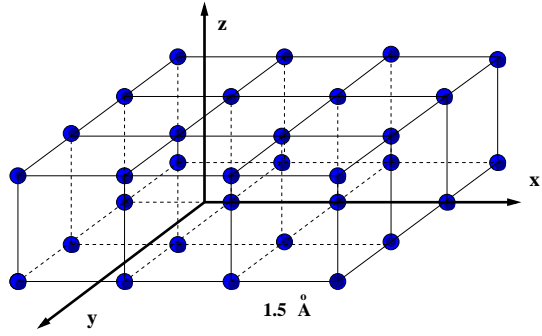}\quad \quad
\caption{\small  Example of the 3D lattice-type structure of size $(4, 4, 2)$. 
}
\label{fig:3DPeriodStruct}  
\end{figure}

Following \cite{VeBoKh:Ewald:14} we introduce the periodic cell 
${\cal R}=  \mathbb{Z}^d$, $d=1,2,3$ for the $\bf k$ index, and consider a 3D $B$-periodic 
supercell $\Omega_L= B\times B\times B$, with $B= \frac{b}{2}[-L,L]$. 
The total electrostatic potential in the supercell $\Omega_L$ is obtained by, first,  
the lattice summation of the Coulomb potentials over  $\Omega_L$ for (rather large) $L$, 
but restricted to the central unit cell $\Omega_0$, and then by replication of the resultant
function to the whole supercell. 
Hence, that the total potential sum $v_{c_L}(x)$ is designated at each elementary 
unit-cell in $\Omega_L$ by the same value (${\bf k}$-translation invariant).
The effect of the conditional convergence of
the lattice summation can be treated by using the  extrapolation 
to the limit (regularization) on
a sequence of different lattice parameters $L$ as described in \cite{VeBoKh:Ewald:14}.

The electrostatic potential in any of $B$-periods can be obtained by copying the 
respective data from $\Omega_L$. The basis set in $\Omega_L$ is constructed by replication 
from the the master unit cell $\Omega_0$ over the whole periodic lattice.

Consider the case $d=3$ in more detail.  
%Supposing for simplicity that $L$ is odd, $L=2p+1$, 
Recall that the reference value $v_{c_L}(x)$ will be computed at the central cell
$\Omega_0$, indexed by $(0,0,0)$, by summation over all contributions from $L^3$ 
elementary sub-cells in $\Omega_L$. For technical reasons here and in the following we
vary the summation index in $k_\ell=0,..., L-1$ and obtain
\begin{equation}\label{eqn:EwaldSumP}
 v_0(x)=  \sum_{\nu=1}^{M_0} \sum\limits_{k_1,k_2,k_3=0}^{L-1}
 \frac{Z_\nu}{\|{x} -a_\nu (k_1,k_2,k_3)\|}, 
\quad x\in \Omega_0.
\end{equation}
The local lattice sum on the index set $n\times n \times n$ corresponding to 
$\Omega_0$, is represented by
\[
 {\bf P}_{\Omega_0} %Z_\nu \sum\limits_{k_1,k_2,k_3=1}^L  {\cal W}_{\nu({\bf k})}  {\bf P}_{\Omega_0}
= \sum_{\nu=1}^{M_0} Z_\nu \sum\limits_{k_1,k_2,k_3=0}^{L-1} \sum\limits_{q=1}^{R_{\cal N}}
{\cal W}_{\nu({\bf k})} \widetilde{\bf p}^{(1)}_{q} 
\otimes \widetilde{\bf p}^{(2)}_{q} \otimes \widetilde{\bf p}^{(3)}_{q}
\in \mathbb{R}^{n\times n \times n},
\]
for the corresponding local projected tensor of small size $n\times n \times n$.
Here  the $\Omega$-windowing operator,
$
{\cal W}_{\nu({\bf k})}={\cal W}_{\nu(k_1)}^{(1)}\otimes {\cal W}_{\nu(k_2)}^{(2)}
\otimes {\cal W}_{\nu(k_3)}^{(3)},
$
restricts onto the small $n\times n \times n$ unit cell by shifting by the lattice vector 
${\bf k}=(k_1,k_2,k_3)$. This reduces both the computational and storage costs by factor $L$.

In the 3D case, we set $q=3$ in the notation for multilevel BC matrix.
Similar to the case of one-level BC matrices, we notice that a matrix 
$A\in {\cal BC} (3,{\bf L},m)$ of size $|{\bf L}| m \times  |{\bf L}| m$ is completely 
defined  by a $3$-rd order coefficients tensor ${\bf A}=[A_{k_1 k_2 k_3}]$ of size 
$L_1 \times L_2 \times L_3 $,
($k_\ell=0,...,L_\ell-1$, $\ell=1,2,3$), with $m\times m$ block-matrix entries,  
obtained by folding of the generating first column vector in $A$. 

\begin{lemma}\label{lem:SparseCaseP}
Assume that in case (P) the number of overlapping unit cells (in the sense of supports of
basis functions) in each spatial direction does not exceed $L_0$. 
Then the Galerkin matrix $V_{c_L}$ exhibits the symmetric, three-level block circulant 
Kronecker tensor-product form. i.e. $V_{c_L} \in {\cal BC} (3,{\bf L},m_0)$, 
(${\bf L}=(L_1,L_2,L_3)$)
\begin{equation}\label{eqn:BC-Core}
V_{c_L}= \sum\limits_{k_1=0}^{L_1-1} \sum\limits_{k_2=0}^{L_2-1} \sum\limits_{k_3=0}^{L_3-1}
\pi_{L_1}^{k_1}\otimes \pi_{L_2}^{k_2}\otimes \pi_{L_3}^{k_3}\otimes A_{k_1 k_2 k_3}, 
\quad A_{k_1 k_2 k_3}\in \mathbb{R}^{m_0\times m_0},
\end{equation}
where the number non-zero matrix blocks $A_{k_1 k_2 k_3}$ does not exceed $(L_0+1)^3$.

The required storage is bounded by $m_0^2 [(L_0 + 1)]^3$ independent of $L$.
The set of non-zero generating matrix blocks $\{ A_{k_1 k_2 k_3}\}$ can be calculated 
in $O(m_0^2 [(L_0 + 1)]^3 n)$ operations.

Furthermore, assume that the QTT ranks of the assembled canonical vectors do not exceed $r_0$. 
Then the numerical cost can be reduced to the logarithmic scale, $O( m_0^2 [(L_0 + 1)]^3\log n)$.
\end{lemma}
\begin{proof}
First, we notice that the shift invariance property in the matrix 
$V_{c_L}=\{V_{{\bf k}{\bf m}}\}$  is a consequence of the
translation invariance in the canonical tensor ${\bf P}_{c_L}$ (periodic case), 
and in the basis-tensor ${\bf G}_{\bf k}$ (by construction),
\begin{equation} \label{eqn:Basis_shift}
{\bf G}_{\bf k m}:= {\bf G}_{\bf k} \odot {\bf G}_{\bf m}= {\bf G}_{|{\bf k} -{\bf m}|} \quad 
\mbox{for} \quad | k_\ell|, |m_\ell| \leq L-1, 
\end{equation}
so that we have
\begin{equation} \label{nuc_BCirculantP}
V_{{\bf k}{\bf m}} = V_{|{\bf k}-{\bf m}|}, \quad 0\leq k_\ell, m_\ell \leq L-1.
%= \langle {\bf G}_{\bf k} \odot {\bf G}_{\bf m} ,   {\bf P}_{\Sigma}\rangle, \quad 1\leq k_\ell, m_\ell \leq L,
\end{equation}
This ensures the perfect three-level block-Toeplitz structure of $V_{c_L}$ 
(compare with the case of a box). 
Now the block circulant pattern in ${\cal BC} (3,{\bf L},m_0)$ is imposed by the 
periodicity of a lattice-structured basis set.

To prove the complexity bounds we observe that a matrix 
$V_{c_L} \in {\cal BC} (3,{\bf L},m_0)$
can be represented in the Kronecker tensor product form (\ref{eqn:BC-Core}),
obtained by an easy generalization of (\ref{eqn:bcircPol}). In fact, we apply 
(\ref{eqn:bcircPol}) by successive slice-wise and fiber-wise splitting
to obtain
\[
\begin{split}
 V_{c_L}
& = \sum\limits_{k_1=0}^{L_1-1}\pi_{L_1}^{k_1}\otimes {\bf A}_{k_1} \\
&= \sum\limits_{k_1=0}^{L_1-1}\pi_{L_1}^{k_1}\otimes 
\left( \sum\limits_{n_2=0}^{L_2-1} \pi_{L_2}^{k_2}\otimes {\bf A}_{k_1 k_2} \right)\\
&=\sum\limits_{k_1=0}^{L_1-1}\pi_{L_1}^{k_1}\otimes 
\left( \sum\limits_{k_2=0}^{L_2-1} \pi_{L_2}^{k_2}\otimes 
\left(\sum\limits_{k_3=0}^{L_3-1} \pi_{L_3}^{k_3}\otimes {A}_{k_1 k_2 k_3} \right) \right),
\end{split}
\]
where ${\bf A}_{k_1}\in \mathbb{R}^{L_2\times L_3 \times m_0\times m_0}$, 
${\bf A}_{k_1 k_2} \in  \mathbb{R}^{L_3 \times m_0\times m_0}$, and 
$A_{k_1 k_2 k_3}\in \mathbb{R}^{m_0\times m_0}$.
Now the overlapping assumption ensures that the number of non-zero matrix blocks 
$A_{k_1 k_2 k_3}$ does exceed $(L_0+1)^3$.

Furthermore, the symmetric mass matrix, $S_{c_L}=\{{s}_{\mu \nu} \}\in \mathbb{R}^{N_b\times N_b}$, 
for the Galerkin  representation of the identity operator reads as follows, 
\begin{equation*}  \label{Ident_pot}
 {s}_{\mu \nu}=  %\int_{\mathbb{R}^3}  \overline{g}_k(x) \overline{g}_m(x) dx = 
 \langle {\bf G}_\mu , {\bf G}_\nu \rangle 
 =\langle S^{(1)}{\bf g}_\mu^{(1)},{\bf g}_\nu^{(1)} \rangle 
 \langle S^{(2)} {\bf g}_\mu^{(2)},{\bf g}_\nu^{(2)} \rangle 
 \langle S^{(3)} {\bf g}_\mu^{(3)},{\bf g}_\nu^{(3)} \rangle, 
\quad 1\leq \mu, \nu \leq N_b,
\end{equation*}
% where now $\{\overline{g}_k\}$ denotes the piecewise constant representation to 
% the respective Galerkin basis functions, and 
 where $N_b=m_0 L^3$.
It can be seen that in the periodic case the block structure in the basis-tensor  
${\bf G}_{\bf k} $ imposes
 the three-level block circulant structure in the mass matrix $S_{c_L}$ 
\begin{equation}\label{eqn:BC-mass}
S_{c_L}= \sum\limits_{k_1=0}^{L_1-1} \sum\limits_{k_2=0}^{L_2-1} \sum\limits_{k_3=0}^{L_3-1}
\pi_{L_1}^{k_1}\otimes \pi_{L_2}^{k_2}\otimes \pi_{L_3}^{k_3}\otimes S_{k_1 k_2 k_3}, 
\quad S_{k_1 k_2 k_3}\in \mathbb{R}^{m_0\times m_0}.
\end{equation}
By the previous arguments we conclude that 
$S_{k_1 k_2 k_3}=S^{(1)}_{k_1} S^{(2)}_{k_2} S^{(3)}_{k_3}$ implying the rank-$1$
separable representation in (\ref{eqn:BC-mass}).

Likewise, it is easy to see that the stiffness matrix representing the (local) Laplace 
operator in the periodic setting has the similar block circulant structure, 
\begin{equation}\label{eqn:BC-Laplace}
\Delta_{c_L}= \sum\limits_{k_1=0}^{L_1-1} \sum\limits_{k_2=0}^{L_2-1} \sum\limits_{k_3=0}^{L_3-1}
\pi_{L_1}^{k_1}\otimes \pi_{L_2}^{k_2}\otimes \pi_{L_3}^{k_3}\otimes B_{k_1 k_2 k_3}, 
\quad B_{k_1 k_2 k_3}\in \mathbb{R}^{m_0\times m_0},
\end{equation}
where the number non-zero matrix blocks $B_{k_1 k_2 k_3}$ does not exceed $(L_0+1)^3$.
In this case the matrix block $B_{k_1 k_2 k_3}$ admits a rank-$3$ product factorization.

This proves the sparsity pattern of our tensor approximation to $H$.
\end{proof}

In the Hartree-Fock calculations for lattice structured systems we deal with 
the multilevel, symmetric block circulant/Toeplitz matrices, where
the first-level blocks, $A_0,...,A_{L_1-1}$, may have further block structures.
In particular, Lemma \ref{lem:SparseCaseP} shows that 
the Galerkin approximation of the 3D Hartree-Fock core Hamiltonian in periodic setting 
leads to the symmetric, three-level block circulant matrix. 

Figure \ref{fig:3DCoreHamPer} represents the block-sparsity in the core Hamiltonian 
matrix  in a box for $L=8$ (left), and  the rotated matrix profile (right).

\begin{figure}[htbp]
\centering
\includegraphics[width=6.0cm]{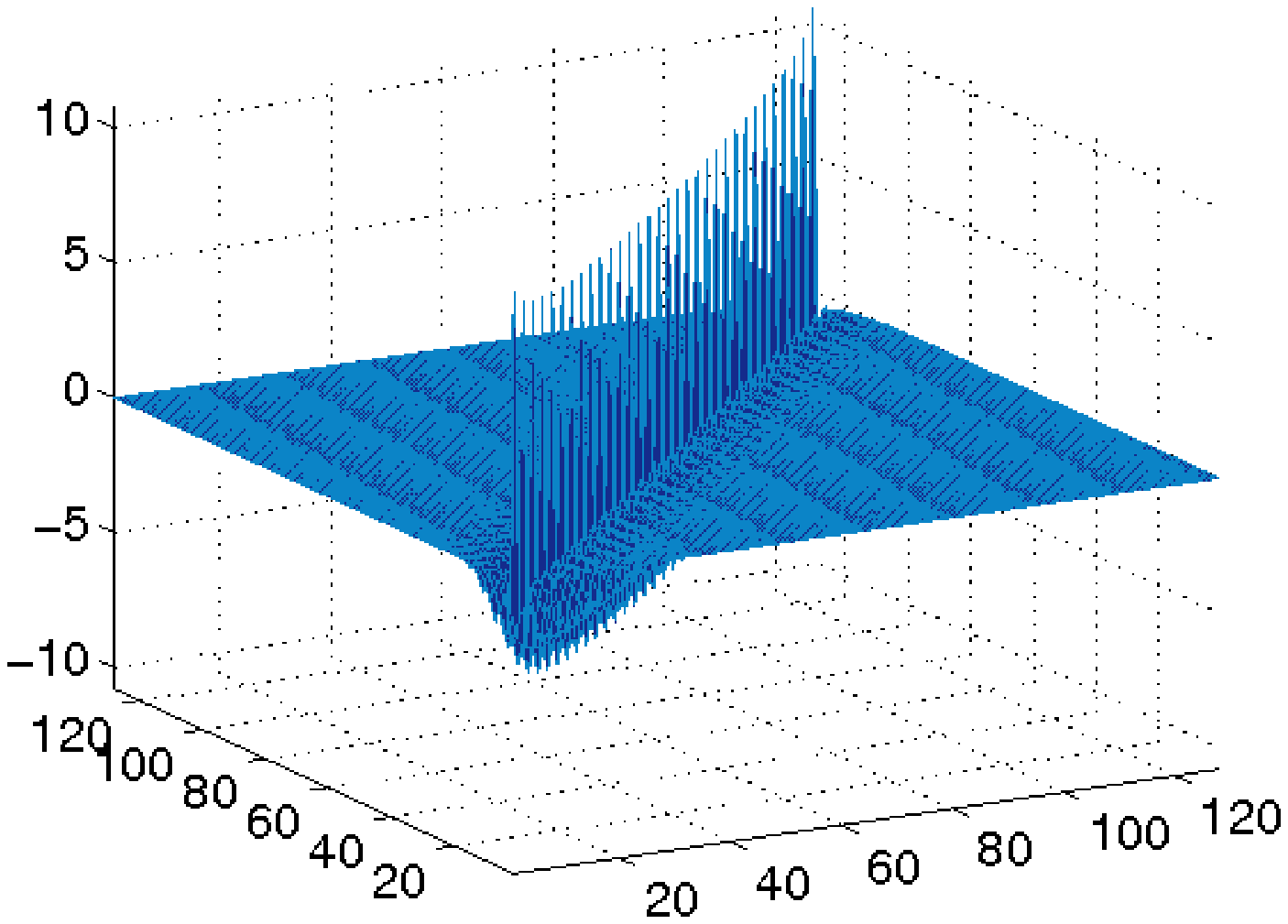}\quad \quad
\includegraphics[width=6.0cm]{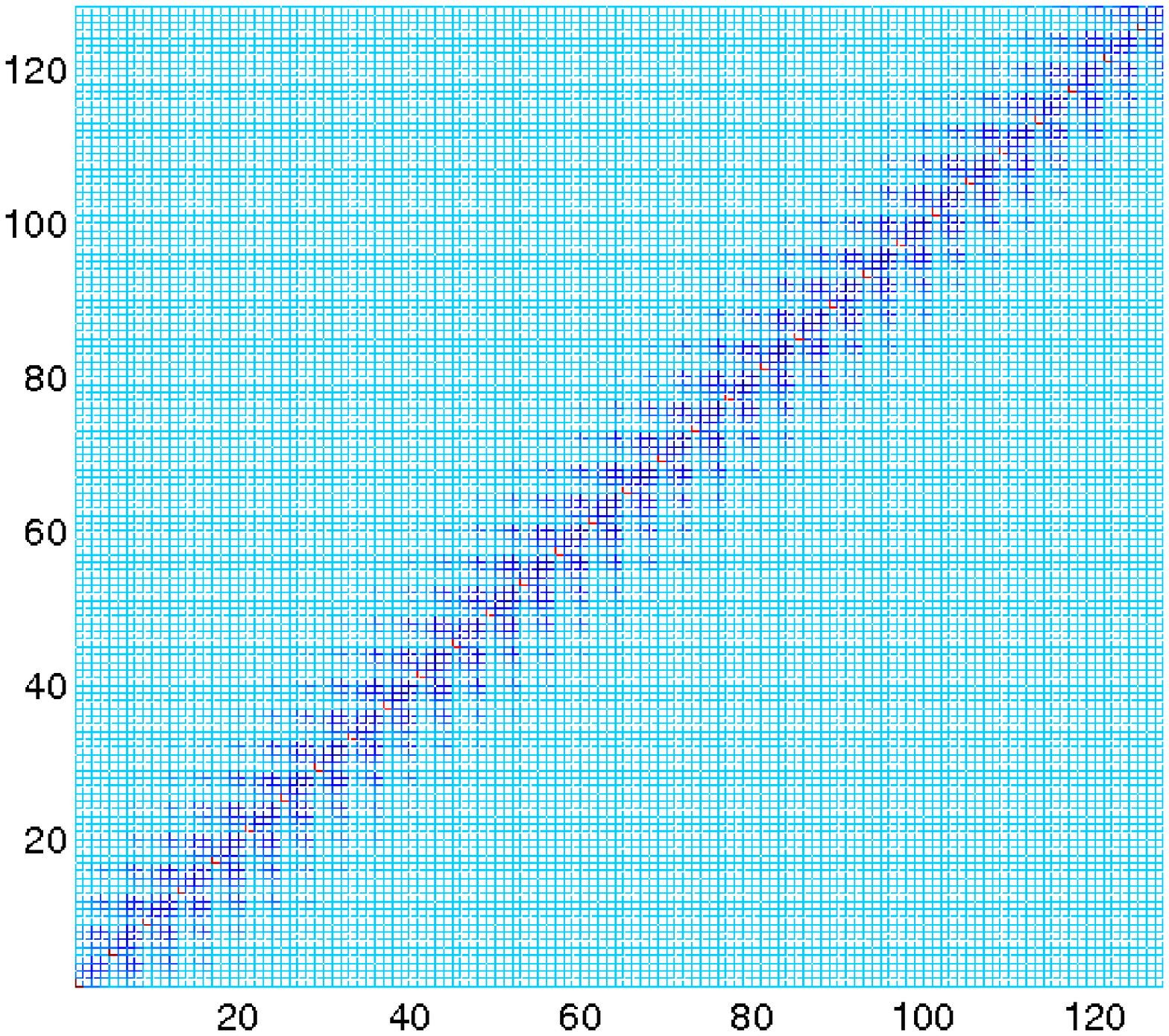}
\caption{\small  Block-sparsity in the core Hamiltonian matrix  in a box for $L=8$ (left); 
Rotated matrix profile (right).}
\label{fig:3DCoreHamPer}  
\end{figure}

In the next section we discuss computational details of the FFT-based eigenvalue solver on
the example of 3D linear chain of molecules. 

\subsection{Regularized spectral problem and complexity analysis}\label{ssec:Complexity_EigPr}

Combining the block circulant representations (\ref{eqn:BC-Core}), (\ref{eqn:BC-Laplace}) 
and (\ref{eqn:BC-mass}),
we are able to represent the eigenvalue problem for the Fock matrix 
in the Fourier space as follows
\begin{equation*}\label{eqn:HF-FSpace}
 \sum\limits^{L_1 -1}_{k_1=0}\sum\limits^{L_2 -1}_{k_2=0}\sum\limits^{L_3 -1}_{k_3=0}
      D_{L_1}^{k_1}\otimes D_{L_2}^{k_2}\otimes D_{L_3}^{k_3}\otimes 
      (B_{k_1 k_2 k_3} + A_{k_1 k_2 k_3}) U
= \lambda 
\sum\limits^{L_1 -1}_{k_1=0}\sum\limits^{L_2 -1}_{k_2=0}\sum\limits^{L_3 -1}_{k_3=0}
D_{L_1}^{k_1}\otimes D_{L_2}^{k_2}\otimes D_{L_3}^{k_3} S_{k_1 k_2 k_3} U, 
\end{equation*}
with the diagonal matrices $D_{L_\ell}^{k_\ell}\in \mathbb{R}^{L_\ell \times L_\ell}$, 
$\ell=1,2,3$, where $U=F_{\bf L}\otimes I_m C$.
The equivalent block-diagonal form reads 
\begin{equation}\label{eqn:HF-FSpace-Block}
 \mbox{bdiag}_{m_0\times m_0} 
 \{{\cal T}_{\bf L}'[F_{\bf L} ({\cal T}_{\bf L}\widehat{B}) 
+ F_{\bf L} ({\cal T}_{\bf L}\widehat{A})] -
\lambda {\cal T}_{\bf L}'(F_{\bf L} [{\cal T}_{\bf L} \widehat{S})]  \} U =0.
%= \lambda\; \mbox{bdiag}_{m_0\times m_0} \{{\cal T}_{\bf n}'(F_{\bf n} 
%[{\cal T}_{\bf n} \widehat{S})]\} U.
\end{equation}
The block structure specified by Lemma \ref{lem:SparseCaseP} allows to apply the efficient
eigenvalue solvers via FFT based diagonalization in the framework of Hartree-Fock calculations,
in general, with the numerical cost $O(m_0^2 L^d \log L)$.
\begin{figure}[htbp]
\centering
\includegraphics[width=6.0cm]{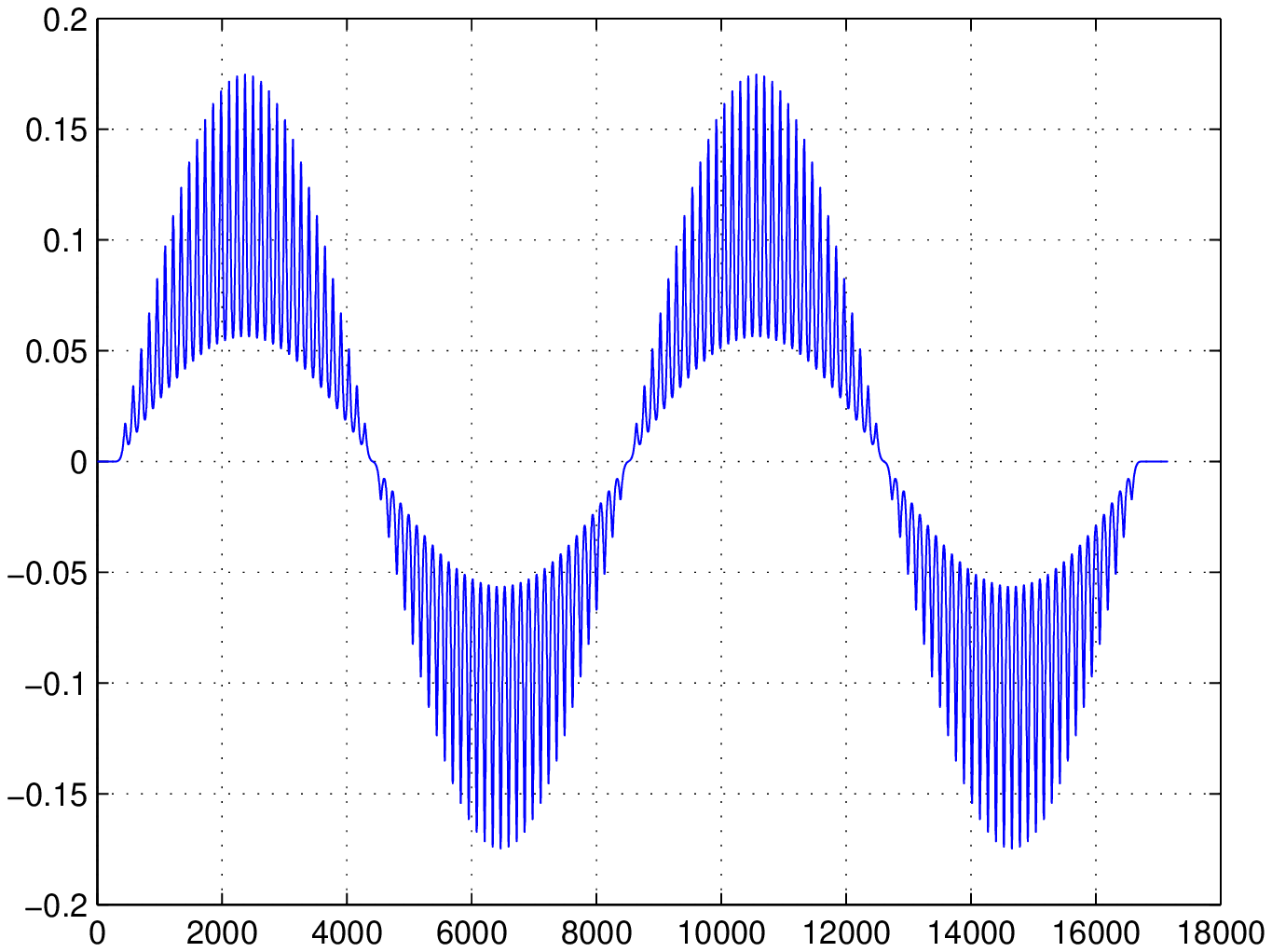} \quad
\includegraphics[width=6.0cm]{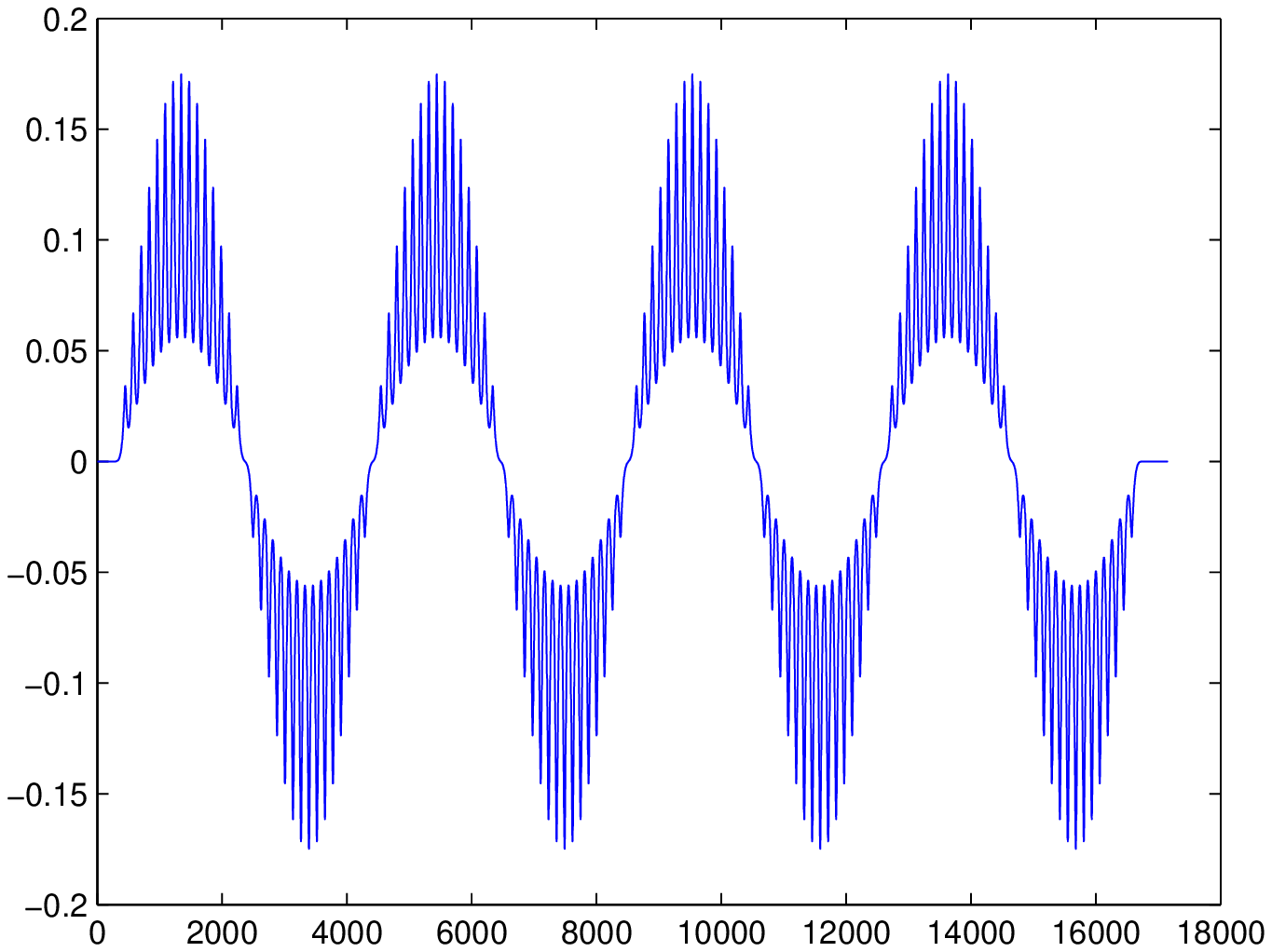}
\caption{\small  Molecular orbitals, i.e. the eigenvectors represented in GTO basis: 
the $4$th orbital (left), the $8$th orbital (right). 
%and the $300$th (right).
}
\label{fig:3DCoreEigVect}  
\end{figure}
\begin{proposition}\label{prop:low_rank_coef}
The low-rank structure in the coefficients tensor mentioned above 
(see Section \ref{ssec:Tensor_bcirc}) allows to reduce the factor $L^d \log L$ to 
$L \log L$ for $d=2,3$. It was already observed in the proof of Lemma \ref{lem:SparseCaseP}
that the respective coefficients in the overlap and Laplacian Galerkin matrices
can be treated as the rank-$1$ and rank-$3$ tensors, respectively.
Clearly, the factorization rank for the nuclear part of the Hamiltonian does not exceed 
$R_{\cal N}$. Hence, Theorem \ref{thm:tens_FFT}  can be applied in generalized form.
\end{proposition} 
Figure \ref{fig:3DCoreEigVect} visualizes molecular orbitals on fine spatial grid 
with $n=2^{14}$: the $4$th orbital (left), the $8$th orbital (right).
The eigenvectors are computed in GTO basis for $(L,1,1)$ system with $L=128$ and $m_0=4$.  

Table \ref{Table_Times_SupSvsPer} compares CPU times in sec. (Matlab) for the 
full eigenvalue solver on a 3D $(L, 1, 1)$ lattice in a box, and for
the FTT-based diagonalization in the periodic supercell, 
all computed for $m_0=4$, $L=2^p$ ($p=7,8,...,15$).
The number of basis function (problem size) is given by $N_b=m_0 L$.
\begin{figure}[htb]
\centering
\includegraphics[width=6.0cm]{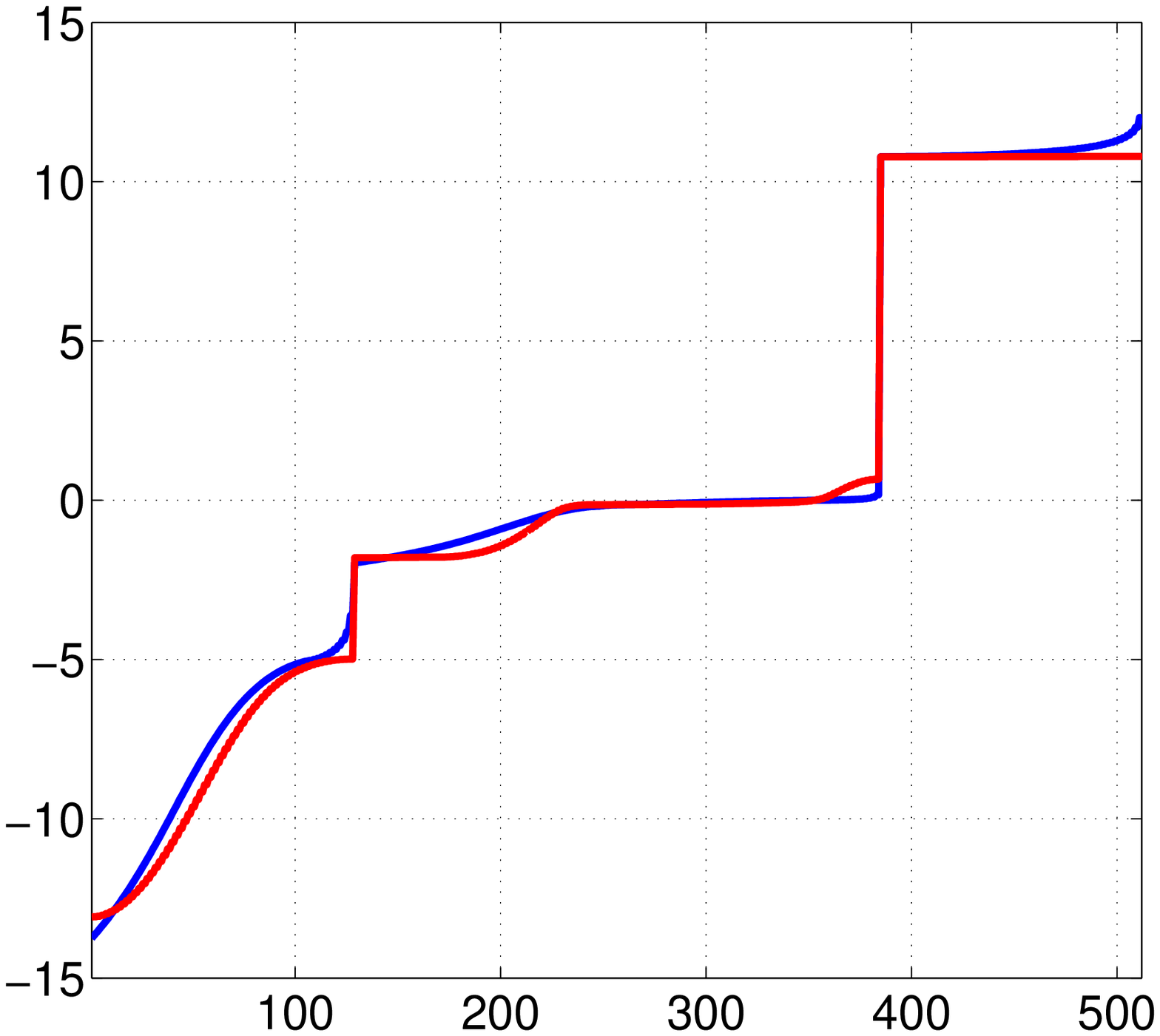}\quad  
\includegraphics[width=6.0cm]{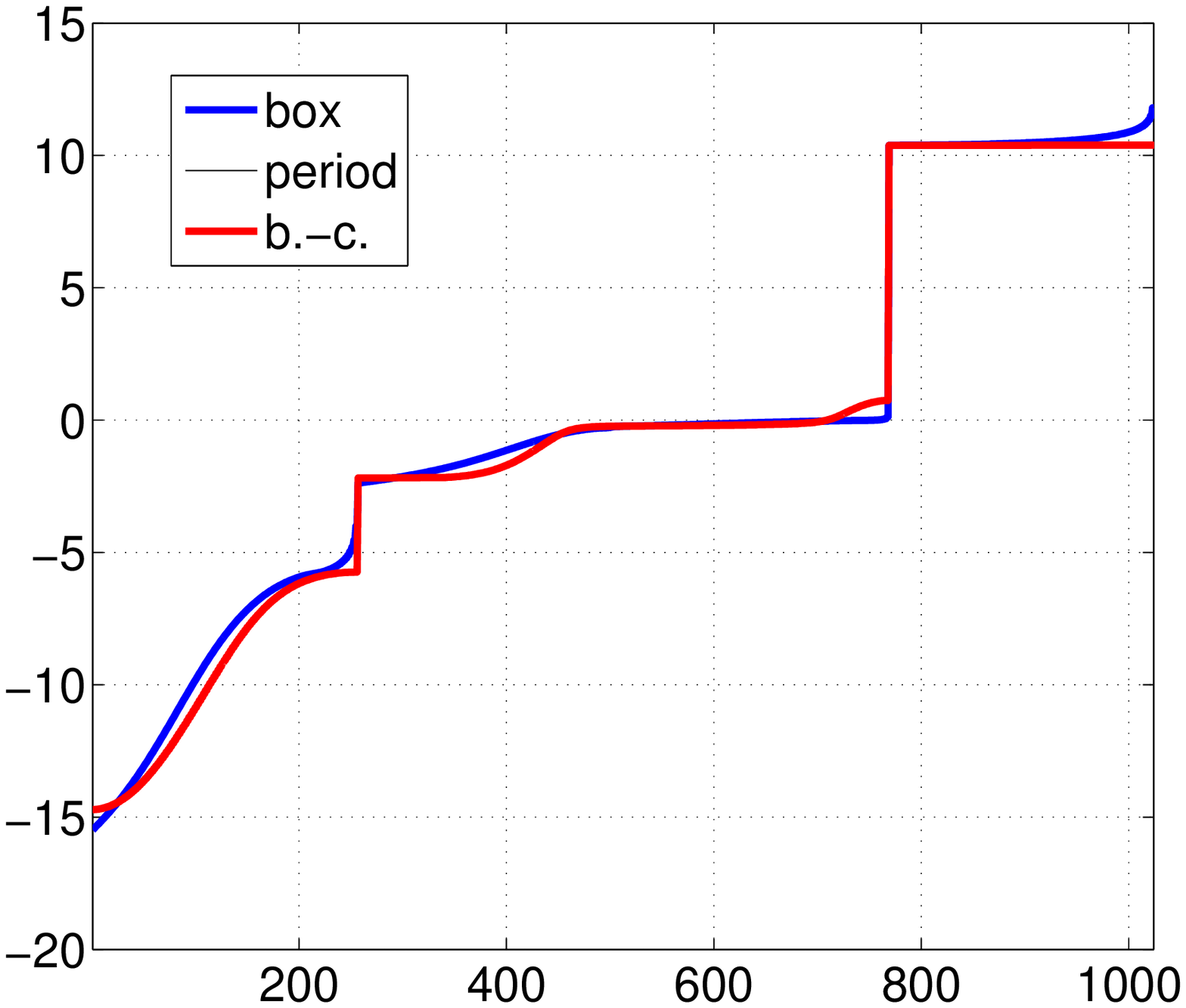}
\caption{\small  Spectrum of the core Hamiltonian.}
\label{fig:3DCoreEig}  
\end{figure}
\begin{table}[htb]
\begin{center}%
\begin{tabular}
[c]{|r|r|r|r|r|r|r|r|r|r|}%
\hline
Problem size $N_b=n_0 L $      & $512$  & $1024$ & $2048$ & $4096$ & $8192$ & $16384$ & $32768$ & $65536$ & $131072$ \\ 
 \hline 
Full EIG-solver         &  $0.67$    &    $5.49$  &   $48.6 $   &  $497.4$    & $--$ & $--$ & $--$ & $--$ & $--$\\
 \hline 
FFT diagonalization  &  $0.10$ &   $0.09$  &   $0.08 $   &  $0.14$    & $0.44$ & $1.5$ & $5.6$ & $22.9$ & $89.4$ \\
 \hline 
% Compression rate  & $2\cdot 10^6$ & $7\cdot 10^6$ & $2\cdot 10^7$ & $1\cdot 10^8$ & $4\cdot 10^8$  \\
%  \hline 
%  mesh size $h$  &  -  & $0.0031$ & $0.0013$ &  $5.9 \cdot 10^{-5} $ & $ $  \\
%  \hline 
 \end{tabular}
\caption{CPU times (sec.): full eig-solver  vs. FFT-based diagonalization for $(L, 1, 1)$ lattice,
and with $m_0=4$, $L=2^p$, $p=7,8,...,15$. }
\label{Table_Times_SupSvsPer}
\end{center}
\end{table}

Figure \ref{fig:3DCoreEig} represents the spectrum of the core Hamiltonian 
in a box vs. those in a periodic supercell for different number of cells $L=128,256$, 
where $m_0=4$.  
The systematic difference between the eigenvalues in both cases can be observed
even for very large $L$. This spectral pollution effects have been discussed 
and theoretically analyzed in \cite{CancesDeLe:08}.
 
\begin{figure}[tbp]
\centering
\includegraphics[width=6.0cm]{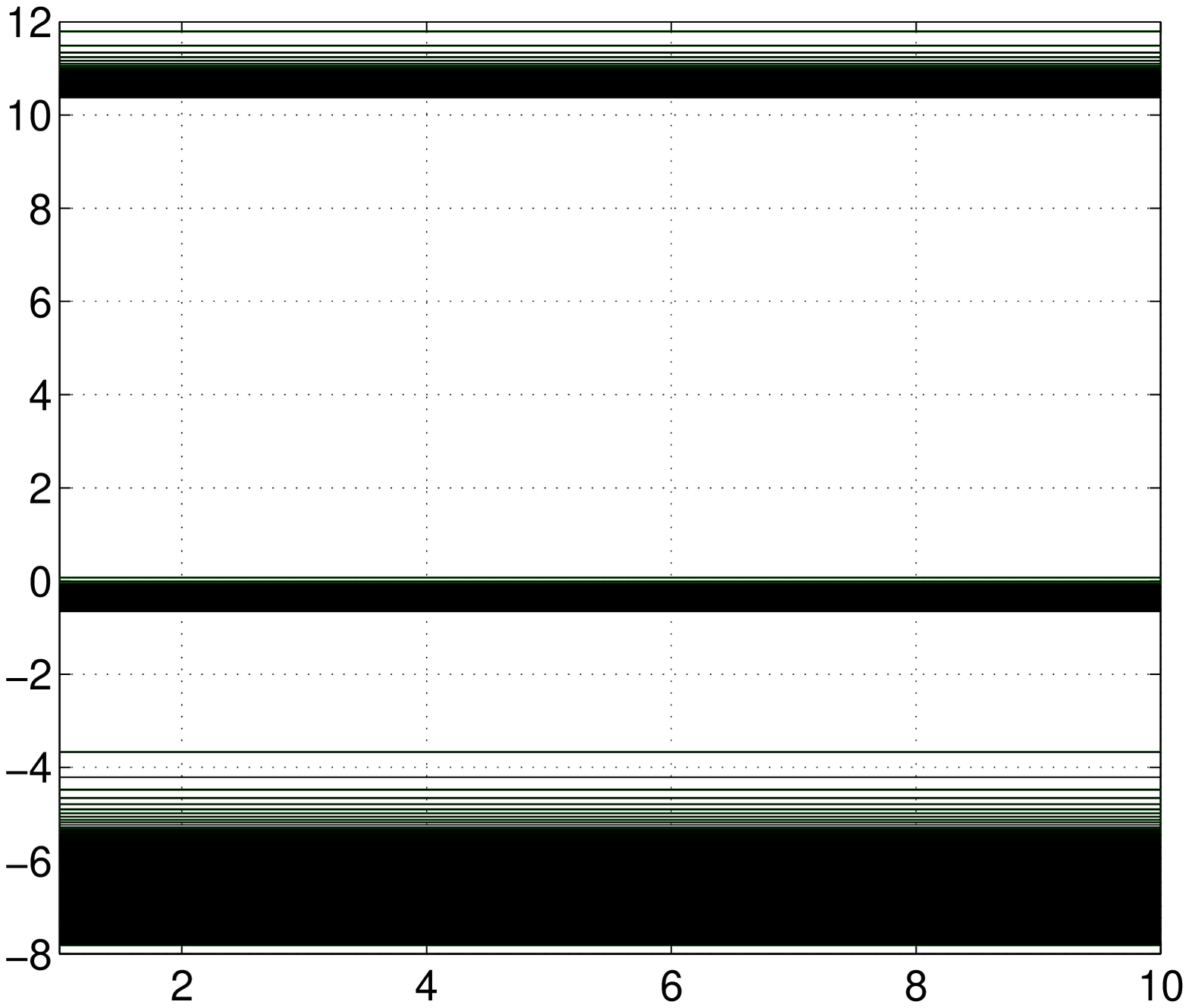}\quad
 \includegraphics[width=6.0cm]{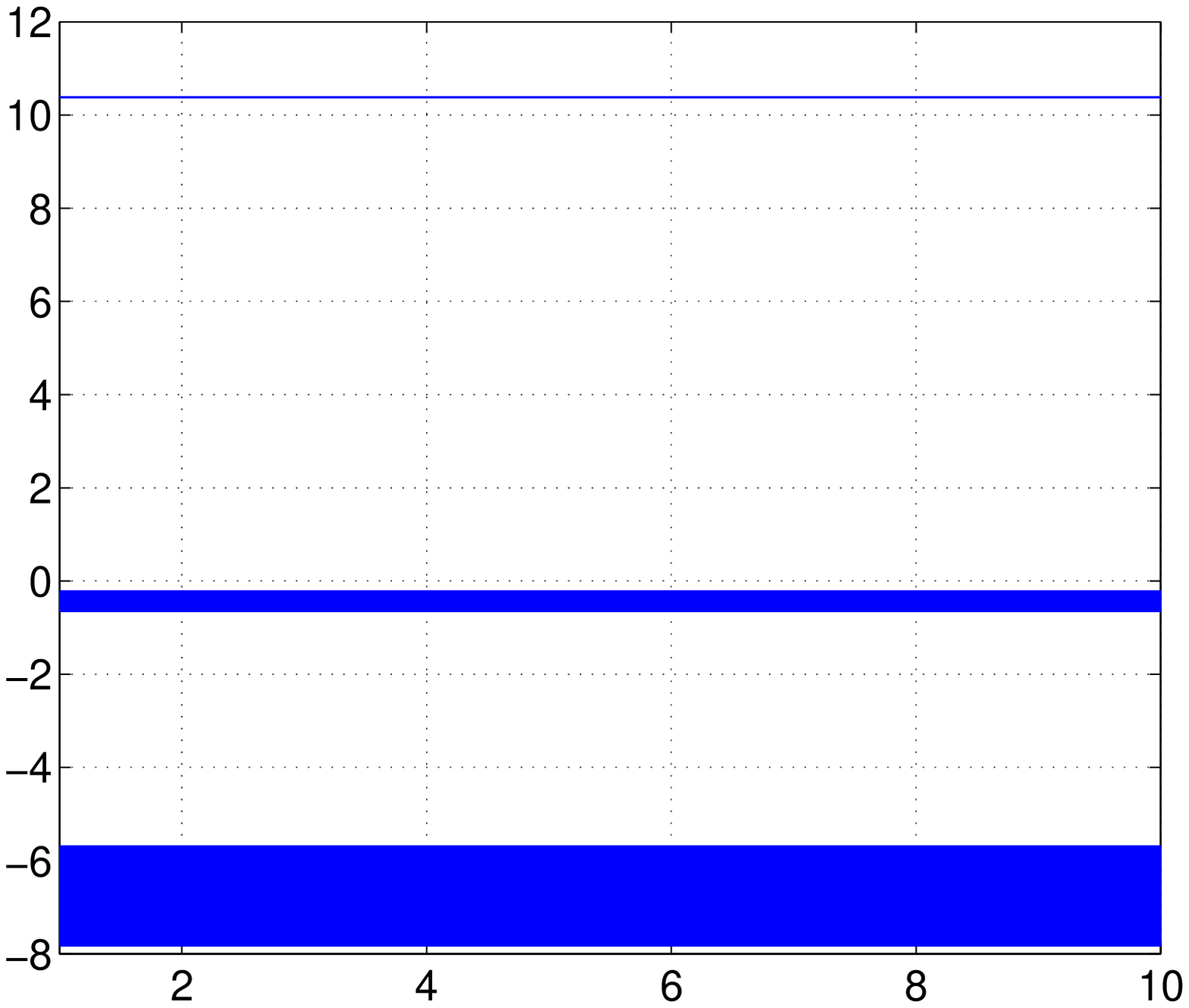}
\caption{\small  Spectrum of the core Hamiltonian for a $(L,1,1)$ lattice with $L=256$, and $m_0=4$,
in a box (left) and for periodic case (right).}
\label{fig:3DSpectBand}  
\end{figure}
Figure \ref{fig:3DSpectBand} presents the spectral bands for 
a $(L,1,1)$ lattice system in a box and in the periodic setting, for $L=256$, and $m_0=4$.
 
\begin{figure}[tbp]
\centering
\includegraphics[width=7.0cm]{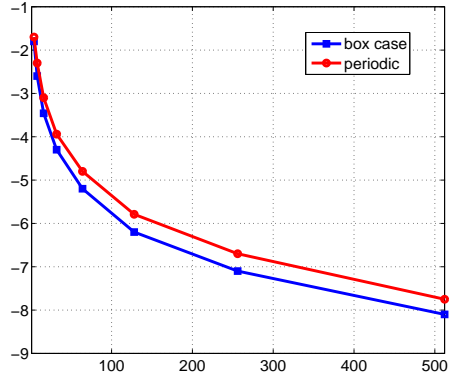} 
\caption{\small  Average energy per unit cell vs. $L$ for a $(L,1,1)$ lattice in a
3D rectangular ``tube``.}
\label{fig:3DAveEig}  
\end{figure}
Figure \ref{fig:3DAveEig} demonstrates the relaxation of the average energy per unit cell
with $m_0=4$, for a $(L,1 ,1)$ lattice structure in a 3D rectangular ``tube`   up
to $L=512$, for both periodic and open boundary conditions.

\section{Conclusions}\label{sec:Conclusions}

We have introduced and analyzed the grid-based tensor product approach to discretization 
and solution
of the Hartree-Fock equation in ab initio modeling of the lattice-structured molecular systems.
In this presentation we consider the case of core Hamiltonian.
All methods and algorithms developed in this paper are implemented and tested in Matlab.

The proposed tensor techniques manifest the twofold benefits: (a)
the entries of the Fock matrix are computed by  1D operations using
low-rank tensors represented on a 3D grid, (b) the low-rank tensor structure in 
the diagonal blocks of the Fock matrix in the Fourier space reduces the conventional 
3D FFT to the product of 1D Fourier transforms. 
The main contributions include:

\begin{itemize}
 \item[$\bullet$] Fast computation of Fock matrix by 1D matrix-vector operations using low-rank tensors represented on 
a 3D spacial grid.
\item[$\bullet$] Analysis and numerical implementation of the multilevel block circulant representation 
of the Fock matrix in the periodic setting.
\item[$\bullet$] Investigation of the low-rank tensor structure in 
the diagonal blocks of the Fock matrix represented in the Fourier space, that
allows to reduce the conventional 3D FFT to the product of 1D FFTs.
\item[$\bullet$] Numerical tests illustrating the computational efficiency of the tensor-structured 
methods applied to the reduced Hartree-Fock equation for lattice-type and periodic systems.
Numerical experiments on verification of the theoretical results on the 
asymptotic complexity estimated of the presented algorithms.
\end{itemize}
Here we confine ourself to the case of core Hamiltonian part in the full Fock matrix (linear part
in the Fock operator). 
The rigorous study of the fully nonlinear self-consistent 
Hartree-Fock eigenvalue problem for periodic and lattice-structured systems in a box
is a matter of future research.   

\section{Appendix: Overview on block circulant matrices}
\label{sec_Append:block-circ}

We recall that a one-level block circulant matrix  $A\in {\cal BC} (L,m_0)$ 
is defined by \cite{Davis},
\begin{equation}\label{eqn:block_c}
A=\operatorname{bcirc}\{A_0,A_1,...,A_{L-1}\}=
 \begin{bmatrix}
A_0     &  A_{L-1} &  \cdots  & A_{2} &  A_{1} \\
A_{1} &  A_0 &  \cdots  & \vdots & A_{2} \\
\vdots  &  \vdots & \ddots &  A_0 & \vdots  \\
A_{L-1}   &  A_{L-2}  & \cdots & A_{1} & A_0 \\
\end{bmatrix}
\in \mathbb{R}^{L m_0\times L m_0},
\end{equation}
where $A_k \in \mathbb{R}^{m_0\times m_0}$ for $k=0,1, \ldots ,L-1$, are matrices of general structure. 
The equivalent Kronecker product representation is defined by the associated matrix polynomial,
\begin{equation}\label{eqn:bcircPol}
 A= \sum\limits^{L -1}_{k=0} \pi^k \otimes A_k =:p_A(\pi),
\end{equation}
where $\pi=\pi_L\in \mathbb{R}^{L \times L}$ is the periodic downward shift 
(cycling permutation) matrix,
\begin{equation}
\pi_L:=
 \begin{bmatrix}
0  &  0    &  \cdots &  0  &  1 \\ 
1  &  0   &  \cdots &  0  &  0 \\
\vdots  &\vdots & \ddots & \vdots & \vdots \\
0  &    \cdots  &  1 &  0  &  0 \\
0  &   0  &  \cdots &  1  &  0 \\
\end{bmatrix}
,
\end{equation}
and $\otimes$ denotes the Kronecker product of matrices. 

In the case $m_0=1$ a matrix $A\in {\cal BC} (L,1)$ defines a circulant matrix 
generated by its first column vector
$\widehat{a}=(a_0,...,a_{L-1})^T$. The associated scalar polynomial then reads
\[
 p_A(z):= a_0 + a_1 z + ... +a_{L-1} z^{L-1},
\]
so that (\ref{eqn:bcircPol}) simplifies to
\[
 A=p_A(\pi_L). 
\]
Let $\omega= \omega_L= \exp(-\frac{2\pi i}{L})$,  we denote by 
$$
F_L=\{f_{k\ell}\}\in \mathbb{R}^{L\times L}, \quad \mbox{with} \quad  
f_{k\ell}=\frac{1}{\sqrt{L}}\omega_L^{(k-1)(\ell-1)},\quad
k,l=1,...,L,
$$
the unitary matrix of Fourier transform. Since the shift matrix $\pi_L$ is diagonalizable 
in the Fourier basis, 
\begin{equation} \label{eqn:diagshift}
 \pi_L=F_L^\ast D_L F_L,\quad D_L= \mbox{diag}\{1,\omega,...,\omega^{L-1} \},
\end{equation}
the same holds for any circulant matrix,
\begin{equation} \label{eqn:circDiag}
 A = p_A(\pi_L) = F_L^\ast p_A(D_L) F_L,  
\end{equation}
where
\[
 p_A(D_L)=\mbox{diag}\{p_A(1),p_A(\omega),...,p_A(\omega^{L-1})\}= \mbox{diag}\{F_L a\}.
\]

Conventionally, we denote by $\mbox{diag}\{x\}$ a diagonal matrix generated by a vector $x$. 
Let $X$ be an $L  m_0\times m_0$ matrix obtained by concatenation of $m_0\times m_0$ 
matrices $X_k$, $k=0,...,L-1$, 
$X=\operatorname{conc}(X_0,...,X_{L-1})=[X_0,...,X_{L-1}]^T$. 
For example, the first block column in (\ref{eqn:block_c})
has the form $\operatorname{conc}(A_0,...,A_{L-1})$.
We denote by $\mbox{bdiag}\{X\}$ the $L m_0\times L m_0$ block-diagonal matrix of 
block size $L$ generated by $m_0\times m_0$ blocks $X_k$. 

It is known that similarly to the case of circulant matrices (\ref{eqn:circDiag}), 
block circulant matrix in ${\cal BC} (L,m_0)$ is unitary equivalent to the block diagonal one 
by means of Fourier transform via  representation (\ref{eqn:bcircPol}), see \cite{Davis}. 
In the following, we describe the block-diagonal representation of a matrix 
$A\in {\cal BC} (L,m_0)$ in the form that is convenient for generalization to the multi-level
block circulants as well as for the description of FFT based implementational schemes. 
To that end, let us introduce the reshaping (folding) transform ${\cal T}_L$ that maps a 
$L m_0\times m_0$ matrix $X$ (i.e., the first block column in $A$) 
to $L\times m_0\times m_0$ tensor $B={\cal T}_L X$ by plugging 
the $i$th $m_0\times m_0$ block in $X$ into a slice $B(i,:,:)$. The respective unfolding returns 
the initial matrix $X={\cal T}_L' B$.
We denote by $\widehat{A}\in \mathbb{R}^{L m_0\times m_0}$ the first 
block column of a matrix $A\in {\cal BC} (L,m_0)$, with a shorthand notation 
$$
\widehat{A}=[A_0,A_1,...,A_{L-1}]^T,
$$ 
so that the $L\times m_0\times m_0$ tensor ${\cal T}_L \widehat{A}$ represents slice-wise
all generating $m_0\times m_0$ matrix blocks.  
%We further use the notation $e_j$ for the $j$th Euclidean basis vector.

\begin{proposition}\label{prop:eig_bcmatr}
For $A\in {\cal BC} (L,m_0)$ we have
\begin{equation} \label{eqn:bcircDiag}
 A= (F_L^\ast \otimes I_{m_0}) \operatorname{bdiag} 
\{ \bar{A}_0, \bar{A}_1,\ldots , \bar{A}_{L-1}\}
(F_L \otimes I_{m_0}),
\end{equation}
where 
\[
 \bar{A}_j = \sum\limits^{L -1}_{k=0} \omega_L^{jk} A_k \in \mathbb{C}^{m_0 \times m_0},
\]
can be recognized as the $j$-th $m_0\times m_0$ matrix block in block column 
${\cal T}_L'(F_L ({\cal T}_L \widehat{A}))$, such that 
\[
  \left[ \bar{A}_0, \bar{A}_1,\ldots , \bar{A}_{L-1}\right]^T = 
 {\cal T}_L'(F_L ({\cal T}_L \widehat{A})).
\]
A set of eigenvalues $\lambda$ of $A$ is then given by
\begin{equation}\label{eqn:lambdAbc}
\{\lambda | Ax = \lambda x, \; x\in \mathbb{C}^{L m_0}\}=
\bigcup\limits_{j=0}^{L-1} \{ \lambda |\bar{A}_j u = \lambda u, \; u \in \mathbb{C}^{m_0}  \}.
 \end{equation}
The eigenvectors corresponding to the spectral sets 
$$
\Sigma_j= \{\lambda_{j, m} |\bar{A}_j u_{j,m} = \lambda_{j,m} u_{j,m}, 
\; u_{j,m} \in \mathbb{C}^{m_0}\}, 
\quad j= 0,1,\ldots , L-1,\quad m=1,...,m_0,
$$ 
can be represented in the form
\begin{equation}\label{eqn:eigvecA}
 U_{j,m}=(F_L^\ast \otimes I_{m}) \bar{U}_{j,m},\quad \mbox{where} \quad 
\bar{U}_{j,m}= E_{[j]} \operatorname{vec}\, [u_{0,m},u_{1,m},...,u_{L-1,m}],
\end{equation}
% where 
% \[
%  \tilde{U}_{jm}= D_j vec\, [u_{0,m},u_{1,m},...,u_{(p-1),m}],
% \]
with $E_{[j]}=\operatorname{diag}\{e_j\}\otimes I_{m_0} \in \mathbb{R}^{L m_0\times L m_0} $, and
$e_j\in \mathbb{R}^{L}$ being the $j$th Euclidean basis vector.
\end{proposition}
\begin{proof}
We combine representations (\ref{eqn:bcircPol}) and (\ref{eqn:diagshift}) to obtain
\begin{align}\label{eqn:Bcircdiag}
A & = \sum\limits^{L -1}_{k=0} \pi^k \otimes A_k = 
      \sum\limits^{L -1}_{k=0} (F_L^\ast D^k F_L) \otimes A_k \\ \nonumber
  & = (F_L^\ast \otimes I_{m_0}) (\sum\limits^{L -1}_{k=0}  D^k  
       \otimes A_k)(F_L \otimes I_{m_0}) \\ \nonumber
  & = (F_n^\ast \otimes I_m)(\sum\limits^{L -1}_{k=0} 
      \mbox{bdiag}\{A_k,\omega_L^k A_k,...,\omega_L^{k(L-1)}A_{k} \}  )  
       (F_L \otimes I_{m_0})\\ \nonumber
  & = (F_L^\ast \otimes I_{m_0})  
      \mbox{bdiag}\{\sum\limits^{L -1}_{k=0} A_k,\sum\limits^{L -1}_{k=0} \omega_L^k A_k,...,
      \sum\limits^{L -1}_{k=0} \omega_L^{k(L-1)}A_{k} \}    (F_L \otimes I_{m_0})\\ \nonumber
  & = (F_L^\ast \otimes I_{m_0}) \mbox{bdiag}_{m_0 \times m_0} 
  \{{\cal T}_L'(F_L ({\cal T}_L \widehat{A}))\} (F_L \otimes I_{m_0}),  \nonumber
\end{align}
where the final step follows by the definition of FT matrix and by the construction of 
${\cal T}_L$. 
The structure of eigenvalues and eigenfunctions  then follows by simple 
calculations with block-diagonal matrices.
\end{proof}

The next statement describes  the block-diagonal form for a class of symmetric BC 
matrices, ${\cal BC}_s (L,m_0)$, 
that is a simple corollary of \cite{Davis}, Proposition \ref{prop:eig_bcmatr}. 
In this case we have $A_0=A_0^T$, and $A_k^T=A_{L-k}$, $k=1,...,L-1$.
\begin{corollary}\label{cor:eig_symbcmatr}
 Let $A\in {\cal BC}_s (L,m_0)$  be symmetric, then $A$ is unitary similar
to a Hermitian block-diagonal matrix, i.e., $A$ is of the form 
\begin{equation}\label{eqn:F_bc}
 A= (F_L \otimes I_{m_0}) \operatorname{bdiag} (\tilde{A}_0, \tilde{A}_1,\ldots , \tilde{A}_{L-1})
(F_L^\ast \otimes I_{m_0}),
\end{equation}
where $I_{m_0}$ is the $m_0\times m_0$ identity matrix.
The matrices $\tilde{A}_j \in \mathbb{C}^{m_0\times m_0}$, $j= 0,1,\ldots , L-1$, are 
defined for even $n\geq 2$ as
\begin{equation}\label{eqn:symBc}
 \tilde{A}_j =A_0 + \sum\limits^{L/2-1}_{k=1} (\omega^{kj}_L A_k + 
\overline{\omega}^{kj}_L A^T_k) + (-1)^j A_{L/2}.
\end{equation}
\end{corollary}

Corollary \ref{cor:eig_symbcmatr} combined with Proposition \ref{prop:eig_bcmatr} describes
a simplified  structure of eigendata in the symmetric case.
Notice that the above representation imposes the symmetry of each 
real-valued diagonal blocks $\tilde{A}_j \in \mathbb{R}^{m_0\times m_0}, \; j= 0,1,\ldots , L-1$,
in (\ref{eqn:F_bc}).

Finally, we recall that a one-level symmetric block Toeplitz matrix  $A\in {\cal BT}_s (L,m_0)$ 
is defined by \cite{Davis},
\begin{equation}\label{eqn:block_SToepl}
A=\operatorname{BToepl}_s\{A_0,A_1,...,A_{L-1}\}=
 \begin{bmatrix}
A_0     &  A_{1}^T &  \cdots  & A_{L-2}^T &  A_{L-1}^T \\
A_{1} &  A_0 &  \cdots  & \vdots & A_{L-2}^T \\
\vdots  &  \vdots & \ddots &  A_0 & \vdots  \\
A_{L-1}   &  A_{L-2}  & \cdots & A_{1} & A_0 \\
\end{bmatrix}
\in \mathbb{R}^{L m_0\times L m_0},
\end{equation}
where $A_k \in \mathbb{R}^{m_0\times m_0}$ for $k=0,1, \ldots ,L-1$, are matrices of a general structure.

\end{document}